\documentclass[12pt,reqno]{amsart}
\usepackage[scale=.86]{geometry}
\usepackage{geometry} 
\usepackage{graphicx}
\usepackage{bbm}
\usepackage{mathrsfs}
\theoremstyle{plain}

\usepackage{CJK}
\usepackage{amsfonts,amssymb,comment,amsmath,mathrsfs,multirow,booktabs}

 \usepackage{hyperref}  

\hypersetup{
	colorlinks = true,  
	linkcolor = blue,  
	citecolor = blue,   
	urlcolor  = red     
}

\usepackage{color,latexsym,amsfonts}
 \newtheorem{theorem}{Theorem}[section]
 \newtheorem{lemma}[theorem]{Lemma}
 \newtheorem{corollary}[theorem]{Corollary}
 \newtheorem{proposition}[theorem]{Proposition}
 \newtheorem{example}[theorem]{Example}
 \newtheorem{Definition}[theorem]{Definition}
 \newtheorem{remark}[theorem]{Remark}
 \newtheorem{condition}[theorem]{Condition}

 \def\beqlb{\begin{eqnarray}}\def\eeqlb{\end{eqnarray}}
 \def\beqnn{\begin{eqnarray*}}\def\eeqnn{\end{eqnarray*}}

 \def\<{\langle}\def\>{\rangle}

 \def\eqref#1{{\rm(\ref{#1})}}

 \def\blue{\color{blue}}

\def\<{\left<}\def\>{\right>}

\newfam\msbmfam\font\tenmsbm=msbm10\textfont
\msbmfam=\tenmsbm\font\sevenmsbm=msbm7
\scriptfont\msbmfam=\sevenmsbm\def\bb#1{{\fam\msbmfam #1}}

\def\PP{\bb P}

\def\<{\left<}\def\>{\right>}

\def\({\left(}\def\){\right)}

\usepackage[dvipsnames]{xcolor}

\newcommand{\R}{\mathbb{R}}
\newcommand{\E}{\mathbb{E}}

\newcommand{\HC}{\mathcal{H}}

\newcommand{\stepid}[1]{\mathbf{1}_{[0, #1]}}

\newcommand{\norm}[1]{\left\lVert#1\right\rVert}
\newcommand{\inner}[2]{\langle #1, #2 \rangle}

\begin{document}
	
	
	
	\centerline{\large\textbf{New Berry-Esseen bounds  for parameter}}
	
	\smallskip
	
	\centerline{\large\textbf{estimation of  Gaussian processes observed at high frequency  }}

	\bigskip
	
	\centerline{$^{1,*}$Khalifa Es-Sebaiy,\, $^{2}$Yong Chen}
	
	\medskip
	
	\centerline{\it $^{1}$Department of Mathematics, College of Science, Kuwait University,}\centerline{\it Sabah Al Salem University City
P.O. Box 5969, Safat 13060, Shadadiya, Kuwait }

\centerline{\it $^{2}$School of Big Data, Baoshan University, Baoshan, 678000, Yunnan, China}

    \centerline{\it $^{*}$Corresponding author}
	
	\centerline{E-mails: \it khalifa.essebaiy@ku.edu.kw, zhishi@pku.org.cn}
	
	\bigskip
	
	{\narrower{\narrower

\centerline{\textbf{Abstract}}

\bigskip
The purpose of this paper is to estimate the limiting variance of  asymptotically stationary Gaussian  processes  observed at high frequency, using the second moment  estimator (SME). We study rates of convergence of the central limit theorem for the SME in terms of the total variation, Kolmogorov  and Wasserstein distances, using  some novel techniques and   sharp estimates for cumulants.
 We apply our approach to provide   Berry–Esseen bounds in
 Kolmogorov and Wasserstein distances for  estimators of the drift parameters of   Gaussian Ornstein-Uhlenbeck processes.
 Moreover,  we prove that  most of our  estimates are strictly sharper than  the ones obtained in the existing literature. 
\bigskip



\par}\par}

\bigskip
\section{Introduction}
In recent years, the statistical inference for asymptotically stationary Gaussian models based on discrete observations has been studied by many
researchers, see, for example, \cite{DEKN,EV,SV} and references therein. Moreover, these results were used to  study the rate of convergence
in central limit theorem (CLT) for drift parameter estimation of Gaussian Ornstein-Uhlenbeck   processes.

Let $Z := \{Z_{t},   t\geq0 \}$ be a continuous centered stationary Gaussian process, and let $X := \{X_{t},   t\geq0 \}$ be the asymptotically stationary Gaussian process defined as\begin{equation*}X_t=Z_t-e^{-\theta t}Z_{0}, \quad t\geq0. 
  \end{equation*}
  In this paper we first  estimate the limiting variance $\E\left(Z_{0}^2\right)$ using the second moment  estimators $v_n(Z)$ and $v_n(X)$, defined, respectively, over the observation window $T_n := n \Delta_n$ by 
  \[v_n(Z) := \frac{1}{n} \sum_{i =0}^{n-1}Z_{t_{i}}^{2}\, \mbox{ and }\, v_n(X) := \frac{1}{n} \sum_{i =0}^{n-1}X_{t_{i}}^{2},\] 
where the processes  $Z$ and $X$ are observed at discrete time instants $t_i = i \Delta_n,\, i = 0, \ldots, n$, with  $\Delta_n$ is the step size converging to zero, as $n\rightarrow\infty$.   \\
Next, we consider the estimator  $f\left(v_n(X)\right)$ to estimate $f\left(\E\left(Z_{0}^2\right)\right)$,
where $f$ is a function satisfying some conditions. 
Specifically, we study rates of convergence of the central limit theorem for the SMEs $v_n(Z)$ and $v_n(X)$ in terms of total variation, Kolmogorov  and Wasserstein distances, and 
 for the estimator  $f\left(v_n(X)\right)$ in terms of Kolmogorov  and Wasserstein distances. We prove that  most of our  estimates are strictly sharper than  the ones obtained in the existing literature. Finally, we apply our approach to derive    explicit Kolmogorov  and Wasserstein bounds in
    central limit theorem for 
estimators  of the drift parameters  of the
fractional Ornstein-Uhlenbeck (fOU)   processes of the first kind and of the second kind.

The novelty of this paper is twofold. (i) The consistency and speed of convergence in the total variation and Wasserstein metrics for the  SMEs $v_n(Z)$ and $v_n(X)$, and  
  in the   Wasserstein distance for the  estimator   $f\left(v_n(X)\right)$  are studied in \cite{DEKN}, using Berry-Esseen bounds of the continuous versions of  $v_n(Z)$ and $v_n(X)$.  Here, using new tools inspired  by the recent paper  \cite{AEA25}, we provide bounds that are strictly sharper than  those obtained in \cite{DEKN}, 
  see the forthcoming Theorem \ref{rate-CLT-Vn} and  Theorem \ref{dW-rate-CLT-Vn}. Note that  \cite{AEA25} only studied the case where  $X$ is a   fOU process. (ii)  We first develop some
novel estimates for the Kolmogorov distance, see the forthcoming Lemma \ref{key-lemma1} and  Lemma \ref{key-lemma2}. Then, we use these results   to derive explicit Kolmogorov bounds   for
the normal approximation of  the estimators $v_n(Z)$,  $v_n(X)$ and $f\left(v_n(X)\right)$,  see the forthcoming Theorem \ref{rate-CLT-Vn} and  Theorem \ref{dKol-rate-CLT-Vn}. It is worth mentioning that the Wasserstein and Kolmogorov bounds we obtained are exactly the same.

Statistical inference for   several Gaussian models  based on discrete observations has been studied by many
researchers.  We will recall some of these results but we mention that the below list is not
exhaustive.
\begin{itemize}
\item  \emph{The case when the mesh in time  $\Delta_n\rightarrow0$}:  The recent papers \cite{AAE24, AE-bridge, EA24, EAA23} provided explicit Wasserstein and  Kolmogorov  bounds in CLT for drift parameter estimation  of It\^o-type diffusions. The parameter estimation for several fractional-noise-driven Ornstein-Uhlenbeck processes,  using the second moment method, was recently studied in the papers in~\cite{HNZ,SV}. However, these papers did not provide the rate of convergence of the CLTs studied therein.
      On the other hand, as mentioned previously,   a quantitative central limit theorem using the Wasserstein distance for parameter estimation for a fOU and for asymptotically stationary Gaussian sequences, respectively, has been studied in \cite{AEA25} and \cite{DEKN}.

\item \emph{The case when the mesh in time  $\Delta_n=1$}: The consistency and speed of convergence in the total variation and Wasserstein distances for the  SME of the limiting variance  of general Gaussian sequences, with applications to Gaussian OU models, were recently developed in the papers \cite{DEV,EV}.
\end{itemize}

The rest of this paper is organized as follows. In Section \ref{sec:review} we  recall some important background material from the Malliavin calculus on Wiener space  in order to make the paper self-contained. The main results and their proofs   are the content of  Section \ref{sect-main} . In Section \ref{applications}  we apply the approach proposed to the fractional Ornstein-Uhlenbeck  processes of the first kind and of the second kind, which illustrates our results and method obtained in  Section \ref{sect-main}.

\section{Elements of Malliavin calculus on Wiener space}\label{sec:review}
 In this section, we provide   a brief overview of some useful facts from the
Malliavin calculus  that are essential for the proofs in the present paper.   We refer the interested
reader to~\cite[Chapter
2]{NP-book} and~\cite[Chapter 1]{nualart-book}.

Let $Z=(Z_t)_{\geq 0}$ be a general centered Gaussian process. We first
  identify the   process $Z$
  with an isonormal Gaussian process $X = \{
X(h), h \in \mathcal{H}\}$ for some Hilbert space $\mathcal{H}$ as follows: $X$ is a centered Gaussian family defined a common
probability space $(\Omega, \mathcal{F},  \mathbb{P})$ satisfying, for all
$h_1, h_2 \in \mathcal{H}$,
$\E [ X(h_1) X(h_2) ] = \inner{h_1}{h_2}_\HC$.\\
One can define $\HC$ as the closure of real-valued step functions on
$[0, \infty)$ with respect to the inner product
$\inner{\stepid{t}}{\stepid{s}}_\HC = \E[ Z_t Z_s]$.   Note that
$X(\stepid{t})
\overset{law}{=} Z_t$.\\
The next step involves the \emph{multiple Wiener-It\^o integrals}.
  We
define the multiple Wiener-It\^o integral $I_p$ via the Hermite
polynomials $H_p$. In particular, for $h \in \HC$ with $\norm{h}_\HC
= 1$, and any $p \geq 1$,
\begin{align}
\notag H_p(X(h)) = I_p(h^{\otimes p}).
\end{align}
 By convention, $I_0(c)=c$ for all $c\in\mathbb{R}$.\\
Given an integer $q \geq 2$ the Hilbert
spaces $\HC^{\otimes q}$ and $\HC^{\odot q}$ correspond to the $q$th
\emph{tensor product} and $q$th \emph{symmetric tensor product} of
$\HC$. If $f \in \HC^{\otimes q}$ is given by $f = \sum_{j_1,
\ldots, j_q} a(j_1, \ldots, j_q) e_{j_1} \otimes \cdots e_{j_q}$,
where $(e_{j_i})_{i \in [1, q]}$ form an orthonormal basis of
$\HC^{\otimes q}$, then the symmetrization $\tilde{f}$ is given by
\begin{align}
\notag \tilde{f} = \frac{1}{q!} \sum_{\sigma} \sum_{j_1, \ldots,
j_q} a(j_1, \ldots, j_q) e_{\sigma(j_1)} \otimes \cdots
e_{\sigma(j_q)},
\end{align}
where the first sum runs over all permutations  $\sigma$ of $\{1,
\ldots, q\}$. Then $\tilde{f}$ is an element of $\HC^{\odot q}$. We
also make use of the concept of contraction. The $r$th
\emph{contraction} of two tensor products $e_{j_1} \otimes \cdots
\otimes e_{j_p}$ and $e_{k_1} \otimes \cdots e_{k_q}$ is an element
of $\HC^{\otimes (p + q - 2r)}$ given by
\begin{align}
\notag (e_{j_1} & \otimes \cdots \otimes e_{j_p}) \otimes_r (e_{k_1} \otimes \cdots \otimes e_{k_q}) \\
\label{eq:contraction} =  & \quad \left[ \prod_{\ell =1}^r
\inner{e_{j_\ell}}{e_{k_\ell}} \right] e_{j_{r+1}} \otimes \cdots
\otimes e_{j_q} \otimes e_{k_{r+1}} \otimes \cdots \otimes e_{k_q}.
\end{align}
Fix integers $p, q \geq
1$ as well as $f \in \HC^{\odot p}$ and $g \in \HC^{\odot q}$.
\begin{align}
\label{eq:isometry}  \E [ I_q(f) I_q(g) ] = \left\{
\begin{array}{ll} p! \inner{f}{g}_{\HC^{\otimes p}} & \mbox{ if } p
= q \\ 0 & \mbox{otherwise.} \end{array} \right.
\end{align}
Let $p,q \geq 1$. If $f \in \HC^{\odot p}$ and $g
\in \HC^{\odot q}$ then
\begin{align}
\label{eq:product} I_p(f) I_q(g) = \sum_{r = 0}^{p \wedge q} r! {p
\choose r} {q \choose r} I_{p + q -2r}(f \widetilde{\otimes}_r g).
\end{align}
For every $q\geq 1$, ${\mathcal{H}}_{q}$ denotes the $q$th Wiener
chaos of $X$, defined as the closed linear subspace of $L^{2}(\Omega
)$ generated by the random variables $\{H_{q}(X(h)),h\in
{{\mathcal{H}}},\Vert h\Vert _{{\mathcal{H}}}=1\}$ where $H_{q}$ is
the $q$th Hermite polynomial. For any $F \in
\oplus_{l=1}^{q}{\mathcal{H}}_{l}$ (i.e. in a fixed sum of Wiener
chaoses), we have the  hypercontractivity property:
\begin{equation}
\left( \E\big[|F|^{p}\big]\right) ^{1/p}\leqslant c_{p,q}\left(
\E\big[|F|^{2}\big]\right) ^{1/2}\ \mbox{ for any }p\geq 2.
\label{hypercontractivity}
\end{equation}
Recall that, for two random variables $X$ and $Y$, the total variation, Kolmogorov  and Wasserstein distances  between the law of $X$ and the law of $Y$   are, respectively, defined by
\begin{align}
\label{eq:def_tv} d_{TV}\left( X,Y\right) := \sup_{A\in \mathcal{B}({\mathbb{R}})}\left\vert \PP\left[ X\in A\right] -\PP\left[ Y\in A\right] \right\vert,
\end{align}
where the supremum is over all Borel sets,  
\begin{equation*}
d_{Kol}\left( X,Y\right):= \sup_{z\in \mathbb{R}}\left\vert
\mathbb{P}\left(X\leq z\right)-\mathbb{P}\left( Y\leq
z\right)\right\vert,
\end{equation*}
 and
\begin{align}
\label{eq:def_w} d_{W}\left( X,Y\right) := \sup_{f\in Lip(1)}\left\vert \E [f(X)]-\E [f(Y)]\right\vert,
\end{align}
where $Lip(1)$ is the set of all Lipschitz functions with Lipschitz constant $\leqslant 1$.\\
Note that \begin{equation}d_{Kol}\left( X,Y\right)\leq d_{TV}\left( X,Y\right).\label{dKol:dTV}\end{equation}
 Let $\mathcal{N}\left(0,1\right)$ 
denote the standard normal law.
Consider a sequence $X:X_{n}\in {\mathcal{H}}_{q}$, such that $\E
X_{n}=0$ and $Var\left[ X_{n}\right] =1$ , and suppose $X_{n}$
converges to a normal law in distribution, which is equivalent to
$\lim_{n} \E\left[ X_{n}^{4}\right] =3$. Then  we have the  optimal
estimate for total variation distance $d_{TV}\left(
X_n,\mathcal{N}\left(0,1\right)\right) $, known as the optimal fourth moment theorem,
proved in \cite{NP2015}:  with the sequence $ X$ as above, assuming convergence, there exist two constants $c,C>0$ depending only on the type of the process $F$ but not on $n$, such that
\begin{equation}
c\max \left\{ \E\left[ X_{n}^{4}\right] -3,\left\vert \E\left[ X_{n}^{3}\right] \right\vert \right\} \leqslant d_{TV}\left(X_{n},\mathcal{N}\left(0,1\right)\right) \leqslant C\max \left\{ \E\left[X_{n}^{4}\right] -3,\left\vert \E\left[ X_{n}^{3}\right]\right\vert \right\} . \label{dTV-optimal berry esseen}
\end{equation}
In view of \cite[Remark 2.2]{DEKN}, this optimal estimate also holds with
Wasserstein distance $d_{W}\left( X_n,\mathcal{N}\left(0,1\right)\right) $  as follows: There exist two constants
$c,C>0$ depending only on the sequence $X$ but not on $n$, such that
\begin{equation}
c\max \left\{ \E\left[ X_{n}^{4}\right] -3,\left\vert \E%
\left[ X_{n}^{3}\right] \right\vert \right\} \leqslant d_{W}\left(
X_n,\mathcal{N}\left(0,1\right)\right) \leqslant C\max \left\{ \E\left[
X_{n}^{4}\right] -3,\left\vert \E\left[ X_{n}^{3}\right] \right\vert
\right\}.\label{dW-fourth-cumulant-thm}
\end{equation}
Let us also recall that   the third and fourth
cumulants are defined, respectively, by
$$
\begin{aligned}
&\kappa_{3}(F)=\E\left[F^{3}\right]-3 \E\left[F^{2}\right] \E[F]+2 \E[F]^{3}, \\
&\kappa_{4}(F)=\E\left[F^{4}\right]-4 \E[F] \E\left[F^{3}\right]-3
\E\left[F^{2}\right]^{2}+12 \E[F]^{2} \E\left[F^{2}\right]-6 \E[F]^{4}.
\end{aligned}
$$
In particular, when $\E[F]=0$, we have  
\[\kappa_{3}(F)=\E\left[F^{3}\right]\ \mbox{ and }\ \kappa_{4}(F)=
\E\left[F^{4}\right]-3 \E\left[F^{2}\right]^{2}.\]
 If
${h\in\mathcal{H}^{\otimes 2}}$, then, according to  \cite[Remark 5.2]{NP10} (or (6.2) and (6.6) in
\cite{BBNP}, respectively),  the third   and fourth cumulants
for $I_2(h)$ satisfy the following estimates:
\begin{eqnarray}
\kappa_{3}(I_2(h))  =\E[(I_2(h))^3] =8\left<h,h\otimes_1
h\right>_{\mathcal{H}^{\otimes 2}},\label{3rd-cumulant}
\end{eqnarray}
and
\begin{eqnarray} \left|\kappa_{4}(I_2(h))\right|
&=&16\left(\|h\otimes_1 h\|_{\mathcal{H}^{\otimes
2}}^2+2\|h\widetilde{\otimes_1}h \|_{\mathcal{H}^{\otimes
2}}^2\right)\nonumber\\&\leq&48\|h\otimes_1
h\|_{\mathcal{H}^{\otimes 2}}^2.\label{4th-cumulant}
\end{eqnarray}

Throughout the paper we use  the notation $C$ for any positive real constant, independently of its value which may change from line to line when this does not lead to ambiguity.

\section{\textbf{Main results}}\label{sect-main}

In this  section, we are interested in the parametric estimation of the variance  of stationary Gaussian process which is not necessarily a semimartingale.

 Let $Z := \{Z_{t},   t\geq0 \}$ be a continuous centered stationary Gaussian process that can be represented as a Wiener-It\^o (multiple) integral $Z_{t} = I_{1}\left(\mathbf{1}_{[0, t]}\right)$ for every $t\geq 0$.  Let $\rho(r)=\E(Z_rZ_0)$ denote the covariance of $Z$ for every $r\geq0$, and  let $\rho(r)=\rho(-r)$ for all $r<0$. We will also consider the non-stationary process $X := \{X_{t},   t\geq0 \}$ defined by
  \begin{equation}X_t=Z_t-e^{-\theta t}Z_{0}, \quad t\geq0.\label{X-Z}
  \end{equation}

In view of \cite[Lemma A.6]{NZ}, we have the following useful inequality.  
\begin{lemma}[\cite{NZ}]\label{NZ-lemma} Fix an integer $M \geq 2.$ We have
\[
\sum_{\left|k_{j}\right| \leq n \atop 1 \leq j \leq
M}|\rho(\mathbf{k} \cdot \mathbf{v})|
\prod_{j=1}^{M}\left|\rho\left(k_{j}\right)\right| \leq
C\left(\sum_{|k| \leq n}|\rho(k)|^{1+\frac{1}{M}}\right)^{M}
\]
where $\mathbf{k}=\left(k_{1}, \ldots, k_{M}\right)$ and $\mathbf{v}
\in \mathbb{R}^{M}$ is a fixed vector whose components are 1 or -1.
\end{lemma}

  Our main assumption throughout the paper is:\\
\noindent \textbf{Assumption}  $\mathbf{\left(\mathcal{A}\right)}$:
We assume that
\begin{align}
\notag \sigma^2 := 4 \int_{\R}\rho^2(r)dr<\infty,
\end{align}
the functions  $|\rho(t)|$ and $|\rho^{\prime}(t)|$ are continuous   on $[0,\infty)$, and   there exist $\gamma>\frac12$ and $m_0>1$  such that, for every $t\geq m_0$,
      \begin{align}|\rho(t)| \leq C t^{-\gamma},\qquad |\rho^{\prime}(t)| \leq C t^{-\gamma},\label{behavior-rho}
      \end{align}
for some constant $C>0$ depending only on $\theta$ and $\rho(0)$.

Our aim is to estimate the limiting variance $\rho(0)=\E\left(Z_{0}^2\right)$ using the second moment  estimator $v_n(X)$, defined over the \emph{observation window} $T_n := n \Delta_n$ by \[v_n(X) := \frac{1}{n} \sum_{i =0}^{n-1}X_{t_{i}}^{2},\] 
where the process $X$ is observed at discrete time instants $t_i = i \Delta_n$, with $i = 0, \ldots, n$ and $\Delta_n$ is the step size.   \\
To proceed we will first estimate  $\rho(0)$  by
 the second moment  estimator $v_{n}\left(Z\right)$, given by  
 \[v_n(Z) := \frac{1}{n} \sum_{i =0}^{n-1}Z_{t_{i}}^{2}.\]
Let us introduce
\begin{align}V_n(Z):=\sqrt{T_n}\left(v_n(Z)-\rho(0)\right),\label{eq:variat_Z}\end{align}
and 
\begin{align}V_n(X):=\sqrt{T_n}\left(v_n(X)-\rho(0)\right).\label{eq:variat_X}\end{align}

We will make use of the following  technical lemmas.
\begin{lemma} Let $V_n(Z)$ and $V_n(X)$ be the sequences given in \eqref{eq:variat_Z}
and \eqref{eq:variat_X}, respectively. Then there exists $C>0$ depending only
on $\theta$ and $\rho(0)$ such that  for every $p \geqslant 1$ and for all $n\geq1$,
\begin{equation}
\left\|V_n(X)-V_n(Z)\right\|_{L^p(\Omega)}\leq
\frac{C}{\sqrt{T_n}}.\label{error X-Z} 
\end{equation}
\end{lemma}
\begin{proof} By \eqref{X-Z}, we can write
\begin{eqnarray*}
\left\|V_n(X)-V_n(Z)\right\|_{L^p(\Omega)}\leq \frac{\sqrt{n\Delta_n}}{n} \sum_{i
=0}^{n-1}\left\|e^{-2\theta t_{i}}Z_{0}^{2}-2e^{-\theta
t_{i}}Z_{0}Z_{t_{i}}\right\|_{L^p(\Omega)}.
\end{eqnarray*}
Next, using the fact that $Z$ is a stationary Gaussian
process, we obtain
\begin{eqnarray*}
\left\|V_n(X)-V_n(Z)\right\|_{L^p(\Omega)}&\leq& \frac{C\sqrt{n\Delta_n}}{n} \sum_{i
=0}^{n-1} e^{-\theta t_{i}}\\
&=& \frac{C\sqrt{n\Delta_n}}{n} \frac{1-e^{-n\theta\Delta_n}}{1-e^{-\theta\Delta_n}}
\\
&\leq& \frac{C}{\sqrt{n\Delta_n}},
\end{eqnarray*}
  where we used   
$\left(1-e^{-\theta x}\right)/{x}\rightarrow \theta$,
as $x\rightarrow0$. Therefore the desired result is
obtained.\end{proof}

\subsection{\textbf{Berry-Esseen bounds   for a second moment estimator  of $\mathbf{\E(Z_0^2)}$}}

Here, we study rates of convergence of the central limit theorem for the SMEs $v_n(Z)$ and $v_n(X)$ in terms of total variation, Kolmogorov  and Wasserstein distances. 

 We will need the following useful lemmas.
\begin{lemma}If the assumption  $\mathbf{\left(\mathcal{A}\right)}$ holds, then, there exists $C>0$ depending only on $\theta$ 
and $\rho(0)$ such that for large
$n$,
\begin{align}
\left|\E\left(V_n^2(Z)\right)-\sigma^2\right|&\leq C\Delta_n +C\left\{ \begin{array}{ll} \frac{1}{T_n} & \text { if } \gamma>1, \\
&\\
\frac{\log(T_n)}{T_n} & \text { if } \gamma=1, \\
&\\
T_n^{1-2\gamma} & \text { if } \frac12<\gamma<1. \end{array} \right.\label{variance-V(Z)}
\end{align}
\end{lemma}

\begin{proof}
Using Wick's formula, we obtain
\begin{align}\E\left(Z_{t}^{2}Z_{s}^{2}\right)=\E\left(Z_{t}^{2}\right)\E\left(Z_{s}^{2}\right)+2\left(\E\left(Z_{t}Z_{s}\right)\right)^2
=\rho^2(0)+2\rho^2(t-s).\label{wick-eq}\end{align} This yields
\begin{align}
\E\left(V_n^2(Z)\right)&=T_n\left[\E v_n^2(Z)-2\rho(0)\E v_n(Z)+\rho^2(0)\right]=T_n\left[\E v_n^2(Z)-\rho^2(0)\right] \nonumber\\
&= T_n\left[\frac{1}{n^2} \sum_{i,j =0}^{n-1} \E\left(Z_{t_{i}}^{2}Z_{t_{j}}^{2}\right)-\rho^2(0)\right] = T_n\left[\frac{2}{n^2} \sum_{i,j =0}^{n-1}
\rho^2\left(t_{j}-t_{i}\right)\right] \nonumber\\
&= \frac{2\Delta_n}{n} \sum_{i,j =0}^{n-1}
\rho^2\left((j-i)\Delta_n\right)=  2\Delta_n \sum_{|k|<n}
\rho^2\left(k\Delta_n\right)\left(1-\frac{|k|}{n}\right)\nonumber\\
&=  2\Delta_n \sum_{|k|<n}
\rho^2\left(k\Delta_n\right)-2\Delta_n \sum_{|k|<n}
\rho^2\left(k\Delta_n\right) \frac{|k|}{n}\nonumber\\
&=-2\Delta_n\rho^2(0)+  4\Delta_n \sum_{k=0}^{n-1}
\rho^2\left(k\Delta_n\right)-4 \sum_{k=1}^{n-1} \frac{k\Delta_n}{n}
\rho^2\left(k\Delta_n\right)\nonumber\\
&=:-2\Delta_n\rho^2(0)+  A_n+B_n.\label{estimate1}
\end{align}
Furthermore,
\begin{align}
  \left|A_n-4\int_{0}^{\infty}\rho^2(x) dx\right|&\leq \left|A_n-4\int_{0}^{T_n}\rho^2(x) dx\right|+4\int_{T_n}^{\infty}\rho^2(x) dx.\label{decomp-An}
\end{align}
Making the change of variable $u=x/T_n$ and using \eqref{behavior-rho}, we can deduce that
\begin{align}
  \int_{T_n}^{\infty}\rho^2(x) dx\leq C T_n^{1-2\gamma}.\label{bound-integ-rho}
\end{align}
On the other hand, by    the Mean Value Theorem, there exist  $\xi_{k,x}\in(t_k,t_{k+1})$, $k=0,\ldots, n-1$, such that
\begin{align*}
  \left|A_n-4\int_{0}^{T_n}\rho^2(x) dx\right|&= 4 \left|\sum_{k=0}^{n-1}\int_{t_k}^{t_{k+1}}\left[\rho^2\left(t_k\right)-\rho^2(x)\right] dx\right|\\
  &\leq 8\sum_{k=0}^{n-1}\int_{t_k}^{t_{k+1}}(x-t_k)\left|\rho\left(\xi_{k,x}\right)\rho'\left(\xi_{k,x}\right)\right| dx\\
     &\leq 8\sum_{\overset{0\leq k\leq n-1}{t_k<m_0}}
  \int_{t_k}^{t_{k+1}}(x-t_k)\left|\rho\left(\xi_{k,x}\right)\rho'\left(\xi_{k,x}\right)\right| dx
 \\&\quad + 8\sum_{\overset{0\leq k\leq n-1}{t_k\geq m_0}} \int_{t_k}^{t_{k+1}}(x-t_k)\left|\rho\left(\xi_{k,x}\right)\rho'\left(\xi_{k,x}\right)\right| dx.
\end{align*}
Moreover, using \eqref{behavior-rho} and $|\rho(t)\rho'(t)|$ is continuous, we get, for every $n>1$ such that $T_n>m_0$,
\begin{align*}
  \left|A_n-4\int_{0}^{T_n}\rho^2(x) dx\right|  
   &\leq  C\sum_{\overset{0\leq k\leq n-1}{t_k<m_0}}
  \int_{t_k}^{t_{k+1}}(x-t_k) dx
  + C\sum_{\overset{0\leq k\leq n-1}{t_k\geq m_0}} t_k^{-2\gamma} \int_{t_k}^{t_{k+1}}(x-t_k)   dx
  \nonumber\\
   &=  C\sum_{0\leq k< \frac{m_0}{\Delta_n}}
  \frac{\Delta_n^2}{2}
  + C\sum_{\frac{m_0}{\Delta_n}\leq k  \leq n-1}  \frac{\Delta_n^2}{2} t_k^{-2\gamma} \nonumber\\ 
   &\leq  C\left( \Delta_n 
  +  \sum_{\frac{1}{\Delta_n}\leq k  \leq n-1}   \Delta_n^2  \left(\lfloor{k\Delta_n}\rfloor\right)^{-2\gamma}  \right).
\end{align*}
Next, making the change of variable $j=\lfloor{k\Delta_n}\rfloor$, then, since $\gamma>\frac12$, we obtain
\begin{align}
  \left|A_n-4\int_{0}^{T_n}\rho^2(x) dx\right|  
   &\leq  C\left( \Delta_n 
  +  \frac{1}{\Delta_n}\sum_{j=1}^{\lfloor{n\Delta_n}\rfloor}   \Delta_n^2 j^{-2\gamma} \right)\nonumber\\
  &\leq  C\Delta_n \left(1 
  +   \sum_{j=1}^{\infty}    j^{-2\gamma}\right)\nonumber\\
  &\leq C\Delta_n.\label{estimate-a_n}
\end{align}
Furthermore,  using \eqref{behavior-rho} and  $\rho^2(t)$ is continuous, we have,  for every    $n>1$ such that $T_n>m_0$, 
\begin{align*}
|B_n|&=4 \sum_{k=1}^{n-1} \frac{k\Delta_n}{n}
\rho^2\left(k\Delta_n\right)\\
&=4 \sum_{1\leq k< \frac{m_0}{\Delta_n}}  \frac{k\Delta_n}{n}
\rho^2\left(k\Delta_n\right)+\sum_{\frac{m_0}{\Delta_n}\leq k  \leq n-1}\frac{k\Delta_n}{n}
\rho^2\left(k\Delta_n\right)\\
&\leq C\left( \frac{1}{n\Delta_n}
+\sum_{\frac{m_0}{\Delta_n}\leq k  \leq n-1}\frac{\lfloor{k\Delta_n}\rfloor}{n}
 \left(\lfloor{k\Delta_n}\rfloor\right)^{-2\gamma}\right).
\end{align*}
Making the change of variable $j=\lfloor{k\Delta_n}\rfloor$, then, in view of  the fact that
  the series $\sum_{j=1}^{\infty}   j^{1-2\gamma} $ is convergent if and only if $\gamma>1$, we have
\begin{align}
|B_n|
&\leq \frac{C}{n\Delta_n}\left( 1
+ \sum_{j=1}^{\lfloor{n\Delta_n}\rfloor}j^{1-2\gamma}\right)\nonumber \\
&\leq  C\left\{ \begin{array}{ll} \frac{1}{T_n} & \text { if } \gamma>1, \\
&\\
\frac{\log(T_n)}{T_n} & \text { if } \gamma=1, \\
&\\
T_n^{1-2\gamma} & \text { if } \frac12<\gamma<1. \end{array} \right..\label{estimate-b_n}
\end{align}
Therefore, combining \eqref{estimate1}, \eqref{decomp-An}, \eqref{bound-integ-rho}, \eqref{estimate-a_n} and \eqref{estimate-b_n},  the estimate  \eqref{variance-V(Z)} follows.
\end{proof}

\begin{lemma}  There exists $C>0$ depending only on $\theta$ and $\rho(0)$   such that for large
$n$,
\begin{align}
|k_{3}(V_n(Z))|&\leq     C\left\{ \begin{array}{ll} \frac{1}{\sqrt{T_n}} & \text { if } \gamma>\frac23, \\
&\\
\frac{\log^2(T_n)}{\sqrt{T_n}}  & \text { if } \gamma=\frac23, \\
&\\
T_n^{\frac32-3 \gamma} & \text { if } \frac12<\gamma<\frac23, \end{array} \right.,\label{k3-estimator}
\end{align}
and
\begin{align} \left|k_{4}(V_n(Z))\right| &\leq  C\left\{ \begin{array}{ll} \frac{1}{T_n} & \text { if } \gamma>\frac34, \\
&\\
\frac{\log^3(T_n)}{T_n}  & \text { if } \gamma=\frac34, \\
&\\
T_n^{2-4 \gamma} & \text { if } \frac12<\gamma<\frac34, \end{array} \right.,.\label{k4-estimator}
\end{align}
As a consequence, \begin{align}
\max\left(|k_{3}(V_n(Z))|,\left|k_{4}(V_n(Z))\right|\right)&\leq     C\left\{ \begin{array}{ll} \frac{1}{\sqrt{T_n}} & \text { if } \gamma>\frac23, \\
&\\
\frac{\log^2(T_n)}{\sqrt{T_n}}  & \text { if } \gamma=\frac23, \\
&\\
T_n^{\frac32-3 \gamma} & \text { if } \frac12<\gamma<\frac23, \end{array} \right.,\label{max-k3-k4-estimator}
\end{align}
\end{lemma}
\begin{proof} 
Using   \eqref{eq:product}, $V_n(Z)$ can be
expressed as
\begin{align}
V_n(Z) = \sqrt{\frac{\Delta_n}{n}} \sum_{i =0}^{n-1}
I_2(\mathbf{1}_{[0, t_i]}^{\otimes 2})=I_2\left(\sqrt{\frac{\Delta_n}{n}}
\sum_{i =0}^{n-1} \mathbf{1}_{[0, t_i]}^{\otimes 2}\right)=:
I_2(\varepsilon_n).\label{V_n(Z) in 2nd chaos}
\end{align}
In view of $\mathbf{1}_{[0, s]}^{\otimes 2} \otimes_{1} \mathbf{1}_{[0,
t]}^{\otimes 2}=\left\langle\mathbf{1}_{[0, s]}, \mathbf{1}_{[0,
t]}\right\rangle_{\mathcal{H}} \mathbf{1}_{[0, s]} \otimes
\mathbf{1}_{[0, t]}=\rho(t-s) \mathbf{1}_{[0, s]} \otimes
\mathbf{1}_{[0, t]},$ we can write
\begin{eqnarray}
\varepsilon_n\otimes_1 \varepsilon_n=\frac{\Delta_n}{n}
\sum_{i,j=0}^{n-1} \rho(t_j-t_i)\mathbf{1}_{[0, t_i]} \otimes
\mathbf{1}_{[0, t_j]}.\label{contaction}
\end{eqnarray}
Using \eqref{3rd-cumulant}, \eqref{V_n(Z) in 2nd
chaos} and \eqref{contaction},  we can deduce 
\begin{eqnarray*}
k_{3}(V_n(Z))=k_{3}(I_2(\varepsilon_n))
&=&8\left<\varepsilon_n,\varepsilon_n\otimes_1
\varepsilon_n\right>_{\mathcal{H}^{\otimes 2}}\nonumber\\
&=&\frac{\Delta_n^{3/2}}{n^{3/2}} \sum_{i,j,k=0}^{n-1}
\rho(t_j-t_i)\rho(t_i-t_k)\rho(t_k-t_j)\nonumber\\
&=&\frac{\Delta_n^{3/2}}{n^{3/2}} \sum_{i,j,k=0}^{n-1}
\rho((i-j)\Delta_n)\rho((i-k)\Delta_n)\rho((j-k)\Delta_n).\nonumber
\end{eqnarray*}
After making the change of variables $i=i, k_1=i-j, k_2=i-k$ and using Lemma \ref{NZ-lemma} with $M=2$ and $v=(-1,1)$, we obtain
\begin{eqnarray}
\left|k_{3}(V_n(Z))\right|
&\leq&\frac{\Delta_n^{3/2}}{n^{1/2}} \sum_{\left|k_1\right|,\left|k_2\right| < n}\left|\rho\left(k_1\Delta_n\right) \rho\left(k_2\Delta_n\right) \rho\left((k_2-k_1)\Delta_n\right)\right|\nonumber \\&\leq& \frac{C\Delta_n^{3/2}}{n^{1/2}}\left(\sum_{|k| < n}|\rho(k\Delta_n)|^{3/2}\right)^2 .\label{k3-estimate1}
\end{eqnarray}
Furthermore, in view of \eqref{behavior-rho} and   $|\rho^2(t)|$ is continuous, we have,  for every   $n>1$ such that $n\Delta_n>m_0$, 
\begin{align*}
\sum_{|k| < n}|\rho(k\Delta_n)|^{3/2}&=\sum_{1\leq k< \frac{m_0}{\Delta_n}} |\rho(k\Delta_n)|^{3/2}+\sum_{\frac{m_0}{\Delta_n}\leq k  \leq n-1} |\rho(k\Delta_n)|^{3/2}\nonumber\\
&\leq C\left(\sum_{1\leq k< \frac{m_0}{\Delta_n}} 1 +\sum_{\frac{m_0}{\Delta_n}\leq k  \leq n-1} |\rho(k\Delta_n)|^{3/2}\right)\nonumber
\\
&\leq C\left(\sum_{1\leq k< \frac{m_0}{\Delta_n}} 1 +\sum_{\frac{m_0}{\Delta_n}\leq k  \leq n-1} |k\Delta_n|^{-\frac32 \gamma}\right)\nonumber
\\
&\leq  C\left( \frac{1}{\Delta_n} +\sum_{\frac{m_0}{\Delta_n}\leq k  \leq n-1} \left(\lfloor{k\Delta_n}\rfloor\right)^{-\frac32 \gamma}\right)\nonumber\end{align*}
Now, making    the change of variable $j=\lfloor{k\Delta_n}\rfloor$ combined with  the fact that
  the series $\sum_{j=1}^{\infty}   j^{-\frac32 \gamma} $ is convergent if and only if $\gamma>\frac23$
\begin{align}\sum_{|k| < n}|\rho(k\Delta_n)|^{3/2}
&\leq  C\left( \frac{1}{\Delta_n} + \frac{1}{\Delta_n} \sum_{j=1}^{\lfloor{n\Delta_n}\rfloor} j^{-\frac32 \gamma}\right)\nonumber\\
 &\leq \frac{C}{\Delta_n}\times \left\{ \begin{array}{ll} 1 & \text { if } \gamma>\frac23, \\
&\\
\log(T_n)  & \text { if } \gamma=\frac23, \\
&\\
T_n^{1-\frac32 \gamma} & \text { if } \frac12<\gamma<\frac23. \end{array} \right.\label{k3-estimate2}
\end{align}
Thus, using \eqref{k3-estimate1} and \eqref{k3-estimate2},  the estimate  \eqref{k3-estimator} holds.\\
On the other hand, using \eqref{4th-cumulant}, \eqref{V_n(Z) in 2nd chaos} and
 \[\notag \mathbf{1}_{[0, s]}^{\otimes 2} \otimes_1 \mathbf{1}_{[0, t]}^{\otimes 2} =
\langle{\mathbf{1}_{[0, s]}},{\mathbf{1}_{[0, t]}}\rangle_{\mathcal{H}} \mathbf{1}_{[0, s]} \otimes \mathbf{1}_{[0, t]}=
\E[ Z_s Z_t]\mathbf{1}_{[0, s]} \otimes \mathbf{1}_{[0, t]},\] we obtain
\begin{eqnarray*} \left|k_{4}(V_n(Z))\right|
&\leq&48\|\varepsilon_n\otimes_1
\varepsilon_n\|_{\mathcal{H}^{\otimes
2}}^2\\&=&48\frac{\Delta_n^2}{n^{2}}\sum_{k_{1},k_{2},k_{3},k_{4}=0}^{n-1}
\left\langle{\mathbf{1}_{[0, t_{k_{1}}]}^{\otimes 2} \otimes_1
\mathbf{1}_{[0, t_{k_{2}}]}^{\otimes 2}}
{\mathbf{1}_{[0, t_{k_{3}}]}^{\otimes 2} \otimes_1 \mathbf{1}_{[0, t_{k_{4}}]}^{\otimes 2}}\right\rangle_{\mathcal{H}^{\otimes 2}}\\
&=&48\frac{\Delta_n^2}{n^{2}}\sum_{k_{1},k_{2},k_{3},k_{4}=0}^{n-1}
\E[ Z_{t_{k_{1}}} Z_{t_{k_{2}}}]\E[ Z_{t_{k_{3}}} Z_{t_{k_{4}}}]\E[
Z_{t_{k_{1}}} Z_{t_{k_{3}}}]\E[ Z_{t_{k_{2}}} Z_{t_{k_{4}}}]
 \\
&=&48\frac{\Delta_n^2}{n^{2}}\sum_{k_{1},k_{2},k_{3},k_{4}=0}^{n-1}
\rho(t_{k_{1}}-t_{k_{2}}) \rho(t_{k_{3}}-t_{k_{4}})
\rho(t_{k_{1}}-t_{k_{3}}) \rho(t_{k_{2}}-t_{k_{4}}),\\
&= &48\frac{\Delta_n^2}{n^{2}}
\sum_{k_{1},k_{2},k_{3},k_{4}=0}^{n-1}\rho((k_{1}-k_{2})\Delta_n)\rho((k_{3}-k_{4})\Delta_n)
\rho((k_{1}-k_{3}))\Delta_n)\rho((k_{2}-k_{4}))\Delta_n).
\end{eqnarray*}
Furthermore, making the change of variables $k_{1}-k_{2}=j_{1},
k_{2}-k_{4}=j_{2}$ and $k_{3}-k_{4}=j_{3}$, and then applying
 Lemma \ref{NZ-lemma}, we get
\begin{eqnarray}\left|k_{4}(V_n(Z))\right|&\leq&48\frac{\Delta_n^2}{n}
\sum_{\underset{i=1,2,3}{\left|j_{i}\right| <
n}}\left|\rho\left(j_{1}\Delta_n\right)
\rho\left(j_{2}\Delta_n\right)
\rho\left(j_{3}\Delta_n\right) \rho\left((j_{1}+j_{2}-j_{3})\Delta_n\right)\right|\nonumber\\
&\leq&C\frac{\Delta_n^2}{n}\left(\sum_{|k| <
n}|\rho(k\Delta_n)|^{\frac{4}{3}}\right)^{3}. \label{k4-estimate1}
\end{eqnarray}%
Moreover, by similar arguments as in the proof
of \eqref{k3-estimate2}, we get, for every   $n>1$ such that $T_n>m_0$, 
\begin{align}
\sum_{|k| <
n}|\rho(k\Delta_n)|^{\frac{4}{3}}
&\leq  C\left( \frac{1}{\Delta_n} + \frac{1}{\Delta_n} \sum_{j=1}^{\lfloor{n\Delta_n}\rfloor} j^{-\frac{4}{3}\gamma}\right)\nonumber\\
 &\leq C\left\{ \begin{array}{ll}  \frac{1}{\Delta_n} & \text { if } \gamma>\frac34, \\
&\\
 \frac{\log(T_n)}{\Delta_n}  & \text { if } \gamma=\frac34, \\
&\\ \frac{T_n^{1-\frac43 \gamma}}{\Delta_n}
 & \text { if } \frac12<\gamma<\frac34. \end{array} \right..\label{k4-estimate2}
\end{align}
Therefore,  using \eqref{k4-estimate1} and \eqref{k4-estimate2}, the estimate \eqref{k4-estimator} follows.
\end{proof}

\begin{lemma}\label{key-lemma1} Define $M_T:=F_T+ G_T$, $T>0$, where, for any $T>0$, $F_T$ and  $G_T$ are in $\mathcal{H}_2$, $\E (F_T^2)\rightarrow c>0$ and $\E (G_T^2)\rightarrow  0$, as $T\rightarrow\infty$. Then, there exists a constant $C>0$ independent of $T$  such that, for  $T$ sufficiently  large,		
		\begin{align} 
			d_{Kol}\left(\frac{M_T}{\sqrt{\E M_T^2}},\mathcal{N}(0, 1)\right) \leq  C d_{Kol}\left(\frac{F_T}{\sqrt{\E F_T^2}},\mathcal{N}(0, 1)\right) +C\|G_T\|_{L^2(\Omega)}.\label{kol-M}
		\end{align}
	\end{lemma}
\begin{proof}
Using  the binomial expansion and \eqref{hypercontractivity}, straightforward calculations yield, for  $T$ sufficiently  large,
\begin{eqnarray}
\left|\kappa_{4}\left(\frac{M_T}{\sqrt{\E M_T^2}}\right)\right|&\leq&C\left|\mathbb{E}(M_T^4)-3(\mathbb{E}(M_T^2))^2\right|\nonumber\\
&\leq&C\left|\mathbb{E}(F_T^4)-3(\mathbb{E}(F_T^2))^2\right|
+C\|G_T\|_{L^2(\Omega)}\nonumber\\
&\leq&\left|\kappa_{4}\left(F_T\right)\right|+C\|G_T\|_{L^2(\Omega)}.\label{4cumulant-M}
\end{eqnarray}
Similarly,  we have, for  $T$ sufficiently  large,
\begin{eqnarray}
\left|\kappa_{3}\left(\frac{M_T}{\sqrt{\E M_T^2}}\right)\right|
&\leq&\left|\kappa_{3}\left(F_T\right)\right|+C\|G_T\|_{L^2(\Omega)}.\label{3cumulant-M}
\end{eqnarray}
Therefore, in view of \eqref{dTV-optimal berry esseen}, \eqref{dKol:dTV}, \eqref{4cumulant-M} and \eqref{3cumulant-M}, we obtain \eqref{kol-M}.
\end{proof}
\begin{theorem}\label{rate-CLT-Vn}
Denote  
\begin{align}\psi_n&:= \Delta_n+\left\{ \begin{array}{ll} \frac{1}{\sqrt{T_n}} & \text { if } \gamma\geq\frac34, \\
&\\
T_n^{1-2 \gamma} & \text { if } \frac12<\gamma<\frac34. \end{array} \right.\label{defi-psi}
\end{align} Assume that the assumption  $\mathbf{\left(\mathcal{A}\right)}$ holds. Then, there exists $C>0$ depending only on $\theta$  and $\rho(0)$  such that for large
$n\geq1$,
\begin{eqnarray}
d_{TV}\left(\frac{1}{\sigma}V_n(X),\mathcal{N}(0,1)\right) 
&\leq&  \frac{C}{T_n^{\frac14}}+C\psi_n,\label{dTV-V(X)}
\end{eqnarray}
\begin{eqnarray}
d_{Kol}\left(\frac{1}{\sigma}V_n(X),\mathcal{N}(0,1)\right) 
&\leq&  \frac{C}{\sqrt{T_n}}+C\psi_n,\label{dKol-V(X)}
\end{eqnarray}
and
\begin{eqnarray}
d_{W}\left(\frac{1}{\sigma}V_n(X),\mathcal{N}(0,1)\right) 
&\leq&  \frac{C}{\sqrt{T_n}}+C\psi_n.\label{dW-V(X)}
\end{eqnarray}
\end{theorem}
\begin{proof} First, we prove  \eqref{dTV-V(X)}. According to  \cite[Theorem 3.5]{kosov},   there is a positive constant $C>0$ such that, for all multiple integrals $F$ and $G$ of order 2,
\begin{eqnarray*}
d_{TV}\left(F,G\right) \leq C \left( \frac{ \E [ (F- G)^2]}{\E[ F^2 ] } \right)^{1/4}.
\end{eqnarray*}
Combining this result with \eqref{error X-Z} and \eqref{variance-V(Z)} yields 
\begin{align}
d_{TV}\left(\frac{1}{\sigma}V_n(X),\mathcal{N}(0,1)\right)&\leq d_{TV}\left(\frac{1}{\sigma}V_n(Z),\mathcal{N}(0,1)\right)
+\frac{1}{\sqrt{\sigma}[\E\left(V_n^2(Z)\right)]^{\frac14}}\left\|V_n(X)-V_n(Z)\right\|^{\frac12}_{L^2(\Omega)}\nonumber\\
&\leq d_{TV}\left(\frac{1}{\sigma}V_n(Z),\mathcal{N}(0,1)\right)+\frac{C}{T_n^{\frac14}}. \label{dTV-V(X)-V(Z)}
\end{align}
Now, we recall the  following technical result \cite[Lemma 5.1]{DEV}:
 Let  $\mu \in \R$ and $\sigma>0$. Then, for every integrable real-valued random variable $F$,
\begin{align}
  d_{TV}\left( \mu+\sigma F,\mathcal{N}\right) &\leq d_{TV}\left( F,\mathcal{N} \right)+ \sqrt{\frac{\pi}{2}}|\mu|+2\left|1-\frac{1}{\sigma^2}\right|.\label{upper-d_TV-lemma}
\end{align}
By \eqref{upper-d_TV-lemma}, we can write
\begin{align}d_{TV}\left(\frac{1}{\sigma}V_n(Z),\mathcal{N}(0,1)\right)&\leq d_{TV}\left(\frac{V_n(Z)}{\sqrt{\E V_n^2(Z)}},\mathcal{N}(0,1)\right)+
2\left|1-\frac{\sigma^2}{\E V_n^2(Z)}\right|.\label{dTV-V(Z)-sigma}
\end{align}
On the other hand, using \eqref{dTV-optimal berry esseen} and \eqref{max-k3-k4-estimator}, we obtain 
\begin{align}
 d_{TV}\left(\frac{V_n(Z)}{\sqrt{\E V_n^2(Z)}},\mathcal{N}(0,1)\right)&\leq C\max \left\{ |k_{3}(V_n(Z))|,\left|k_{4}(V_n(Z))\right| \right\}\nonumber\\
 &\leq C\left\{ \begin{array}{ll} \frac{1}{\sqrt{T_n}} & \text { if } \gamma>\frac23, \\
&\\
\frac{\log^2(T_n)}{\sqrt{T_n}}  & \text { if } \gamma=\frac23, \\
&\\
T_n^{\frac32-3 \gamma} & \text { if } \frac12<\gamma<\frac23. \end{array} \right.\label{dTV-V(Z)}
\end{align}
Therefore, combining \eqref{variance-V(Z)}, \eqref{dTV-V(X)-V(Z)}, \eqref{dTV-V(Z)-sigma} and \eqref{dTV-V(Z)}, the desired result \eqref{dTV-V(X)} is obtained.\\
Let us now prove \eqref{dKol-V(X)}.
 Using \eqref{X-Z}, we can write
\begin{align*}
R_n:=V_n(X)-V_n(Z)= \frac{\sqrt{T_n}}{n}\sum_{i=0}^{n-1}\left(e^{-2\theta t_{i}}Z_{0}^{2}-2e^{-\theta
t_{i}}Z_{0}Z_{t_{i}}\right).
\end{align*}
Combining this with \eqref{eq:product} and \eqref{error X-Z}, we can deduce  that, for every $n\geq1$,
\begin{align}R_n-\E R_n\in \mathcal{H}_2,\, \mbox{ and }
\left\| R_n-\E R_n\right\|_{L^2(\Omega)}\leq
\frac{C}{\sqrt{T_n}}.\label{ineq-R}
\end{align}
Next, we recall that, by straightforward calculations, 
we have,   for every  $\mu \in \R$ and $\sigma>0$ and for every integrable real-valued random variable $F$,
 \begin{align} d_{Kol}\left(\mu+F,\mathcal{N}(0,1)\right) \leq d_{Kol}\left( F,\mathcal{N}(0,1) \right)+\frac{1}{\sqrt{2\pi}}|\mu|.\label{ineq-mean-dKol}
 \end{align}
Using \eqref{ineq-R} and \eqref{ineq-mean-dKol}, and applying Lemma \ref{key-lemma1} for $M_n=V_n(X)$, $F_n=V_n(Z)$ and $G_n=R_n$, we obtain
\begin{align*}
d_{Kol}\left(\frac{1}{\sigma}V_n(X),\mathcal{N}(0,1)\right)&\leq d_{Kol}\left(\frac{1}{\sigma} {(V_n(Z)+R_n-\E R_n)},\mathcal{N}(0,1)\right) {+\frac{1}{\sqrt{2\pi}\sigma}|\E R_n|}\\
&\leq Cd_{Kol}\left(\frac{1}{\sigma}V_n(Z),\mathcal{N}(0,1)\right)  +C {\|R_n-\E R_n\|_{L^2(\Omega)}}+\frac{1}{\sqrt{2\pi}\sigma}|\E R_n|  \\
&\leq  \frac{C}{\sqrt{T_n}}+C\psi_n,
\end{align*}
where we used  \eqref{dTV-optimal berry esseen}, \eqref{dKol:dTV}  and \eqref{max-k3-k4-estimator}. Thus the proof of \eqref{dKol-V(X)} is complete.\\ 
For \eqref{dW-V(X)}, we have
\begin{align}
d_W\left(\frac{1}{\sigma}V_n(X),\mathcal{N}(0,1)\right) &\leq  {\frac{1}{\sigma}}\E \left| V_n(X)-V_n(Z)\right|+d_W\left(\frac{1}{\sigma}V_n(Z),\mathcal{N}(0,1)\right).\label{dW-V(X)-V(Z)}
\end{align}
Using  \eqref{dW-V(X)-V(Z)} and \eqref{dW-fourth-cumulant-thm} combined with similar arguments as above, we obtain
\begin{align*}
d_W\left(\frac{1}{\sigma}V_n(X),\mathcal{N}(0,1)\right)
&\leq C\left(\max\left(|k_{3}(V_n(Z))|,\left|k_{4}(V_n(Z))\right|\right)+\left|\E\left(V_n^2(Z)\right)-\sigma^2\right|
+\frac{1}{\sqrt{T_n}}\right).
\end{align*}
Combining this with  \eqref{variance-V(Z)} and \eqref{max-k3-k4-estimator}, we can deduce the estimate \eqref{dW-V(X)}.
\end{proof}

\subsection{\textbf{Berry-Esseen bounds   for an estimator  of $\mathbf{f\left(\E(Z_0^2)\right)}$}}

Here, we study rates of convergence of the central limit theorem   for the estimator  $f\left(v_n(X)\right)$ in terms of Kolmogorov  and Wasserstein metrics.

Let us first recall that,  by straightforward calculations, we have, for all $u,v\in \mathbb{R}$  and $uv\ge 0$,
\begin{eqnarray}\left|\mathbb{P}\left\{ \mathcal{N}(0,\sigma^2)  \leq
u\right\}-\mathbb{P}\left\{ \mathcal{N}(0,\sigma^2)  \leq
v\right\}\right|\leq\frac{|u-v|}{\sqrt{2\pi}\sigma}e^{-\frac{\min(u^2,v^2)}{2\sigma^2}}, \label{diff-proba-estimate}\end{eqnarray}
\begin{eqnarray}\sup_{\{x>0\}}|x|e^{-x^2}<\infty,\qquad
 \sup_{\{x>0\}}x^2e^{- x^2}<\infty,\label{sup-z-exp}
\end{eqnarray}
and for every $x>0$,
 \begin{align}\mathbb{P}\left\{\mathcal{N}(0,\sigma^2) \leq -x\right\} =\mathbb{P}\left\{\mathcal{N}(0,\sigma^2) \geq x\right\}&=\int_{x}^{\infty}\frac{1}{\sqrt{2\pi}\sigma} e^{-\frac{y^2}{2\sigma^2}}dy\nonumber\\
&\leq e^{-\frac{x^2}{4\sigma^2}}\int_{x}^{\infty} \frac{1}{\sqrt{2\pi}\sigma} e^{-\frac{y^2}{4\sigma^2}}dy \leq C
e^{-\frac{x^2}{4\sigma^2}}.\label{unif-expo}
 \end{align} 

The following useful lemma is needed for this section.
\begin{lemma}\label{key-lemma2}
		Denote $N\sim \mathcal{N}(0, 1)$. Fix $T>0$ and $\theta>0$. Let $F$ be a random variable such that $F>0$ almost surely and $\E F<\infty$, and let $f$ and $g$ be  two functions defined on $(0,\infty)$ such that  $f(F)>0$ almost surely, $g(x)=f^{-1}(x)$,   $g^{\prime}(x)<0$, $g^{\prime\prime}(x)>0$, $g^{\prime\prime\prime}(x)<0$.
		Then,		there exists a constant $C>0$ independent of $T$ and $F$ such that 
		\begin{align} 
			d_{Kol}\left(\sqrt{T}\left(f(F)-\theta\right),N\right) \leq  C  \left[d_{Kol}\left(\frac{\sqrt{T}}{|g^{\prime}(\theta)|}(F-\E F), N\right) +\sqrt{T}\left| \E F - g(\theta)\right|+\frac{1}{\sqrt{T}}\right].\label{kol-ratio}
		\end{align}
	\end{lemma}
\begin{proof} The proof of this estimate is inspired by computations developed in the proof of \cite[Theorem 3.2]{SV}. 
Define 
\[A(z):=\mathbb{P}\left\{\sqrt{T}\left(f(F)-\theta\right)\leq z\right\}-\phi(z), \quad \mbox{with } \phi(z):=\mathbb{P}\left\{N\leq z\right\}.\]
Suppose $z\leq-\sqrt{T}\theta$, so, $\theta+\frac{z}{\sqrt{T}}\leq 0$. Since $f(F)>0$, we have
\begin{align*} \mathbb{P}\left\{\sqrt{T}\left(f(F)-\theta\right)\leq z\right\}= \mathbb{P} \left\{f(F) \leq \theta+\frac{z}{\sqrt{T}}\right\}=0.
\end{align*} 
Thus, in view of \eqref{unif-expo},
\begin{align*}\left|A(z)\right|= \phi(z) \leq  \phi(-\sqrt{T}\theta)\leq  C e^{-\frac{T\theta^2}{4}}\leq \frac{C}{\sqrt{T}}.
\end{align*}
Now suppose $z>-\sqrt{T}\theta$, so, $\theta+\frac{z}{\sqrt{T}}> 0$. Using $g$ is decreasing, we have
\begin{align*} \mathbb{P}\left\{\sqrt{T}\left(f(F)-\theta\right)\leq z\right\}&= \mathbb{P} \left\{f(F) \leq \theta+\frac{z}{\sqrt{T}}\right\}
\\&=\mathbb{P} \left\{\frac{\sqrt{T}}{|g^{\prime}(\theta)|}\left(F-g(\theta)\right) \geq \frac{\sqrt{T}}{|g^{\prime}(\theta)|}\left(g\left(\theta+\frac{z}{\sqrt{T}}\right)-g(\theta)\right)\right\}.
\end{align*}
Denote
\[\nu:=\frac{\sqrt{T}}{|g^{\prime}(\theta)|}\left(g\left(\theta+\frac{z}{\sqrt{T}}\right)-g(\theta)\right),\qquad 
\mu:=\nu-\frac{\sqrt{T}}{|g^{\prime}(\theta)|}\left(\E F-g(\theta)\right).\]
Thus, by \eqref{diff-proba-estimate}, we can write, 
\begin{align*}\left|A(z)\right|&= \left| \mathbb{P} \left\{\frac{\sqrt{T}}{|g^{\prime}(\theta)|}\left(F-\E F\right) \geq\mu\right\}-\phi(z)\right|
\\&\leq \left| \mathbb{P} \left\{\frac{\sqrt{T}}{|g^{\prime}(\theta)|}\left(F-\E F\right) \geq\mu\right\}-\mathbb{P} \left\{N\geq \mu\right\}\right|\\
&\quad +\left|\mathbb{P} \left\{N\geq \mu\right\}-\mathbb{P}\left\{N\geq \nu\right\}\right|+\left|\mathbb{P} \left\{N\geq \nu\right\}-\phi(z)\right|
\\&\leq d_{Kol}\left(\frac{\sqrt{T}}{|g^{\prime}(\theta)|}(F-\E F), N\right)
+ \frac{\sqrt{T}}{|g^{\prime}(\theta)|}\left|\E F-g(\theta)\right|+\left|\mathbb{P} \left\{N\geq \nu\right\}-\phi(z)\right|.
\end{align*}
Therefore, to achieve \eqref{kol-ratio}, it remains to prove that, for every $z>-\sqrt{T}\theta$,
\begin{align}\left|\mathbb{P} \left\{N\geq \nu\right\}-\phi(z)\right|\leq \frac{C}{\sqrt{T}}.\label{phi-nu}
\end{align}
Using the Mean Value Theorem and $|g^{\prime}(\theta)|=-g^{\prime}(\theta)$, we obtain that, for each $z$, there exists $\eta\in(0,1)$ such that 
\begin{align*}\nu=\frac{z}{|g^{\prime}(\theta)|} g^{\prime}\left(\theta+\frac{\eta z}{\sqrt{T}}\right)=\frac{-z}{g^{\prime}(\theta)} g^{\prime}\left(\theta+\frac{\eta z}{\sqrt{T}}\right).\end{align*}
This implies, for every $z$,
\begin{align}\left|\mathbb{P} \left\{N\geq \nu\right\}-\phi(z)\right|=\left|\phi\left(\frac{z}{g^{\prime}(\theta)} g^{\prime}\left(\theta+\frac{\eta z}{\sqrt{T}}\right)\right)-\phi(z)\right|.\label{MVT-eta}
\end{align}
Furthermore, by the Mean Value Theorem, for each $z$, there exists $\delta\in(0,1)$ such that 
\begin{align}\frac{g^{\prime}\left(\theta+\frac{\eta z}{\sqrt{T}}\right)}{g^{\prime}(\theta)}-1& =  \frac{\eta z}{\sqrt{T}g^{\prime}(\theta)}
g^{\prime\prime}\left(\theta+\frac{\eta\delta z}{\sqrt{T}}\right).\label{MVT-delta.}
\end{align}
If $-\sqrt{T}\theta<z<-\frac{\sqrt{T}\theta}{2}$, then,  $0<\theta+\frac{\eta z}{\sqrt{T}}<\theta$, since $-{\blue \theta}<\frac{z}{\sqrt{T}}<\frac{\eta z}{\sqrt{T}}$. So, $\frac{g^{\prime}\left(\theta+\frac{\eta z}{\sqrt{T}}\right)}{g^{\prime}(\theta)}>1$. Combining this with \eqref{unif-expo} and \eqref{MVT-eta} yields
 \begin{align}\left|\mathbb{P} \left\{N\geq \nu\right\}-\phi(z)\right|&= \phi(z)-\phi\left(z\frac{g^{\prime}\left(\theta+\frac{\eta z}{\sqrt{T}}\right)}{g^{\prime}(\theta)} \right) \nonumber \\
& \leq \phi(z)\nonumber\\
& \leq  \phi\left(-\frac{\sqrt{T}\theta}{2}\right)\leq  C e^{-\frac{T\theta^2}{16}}\leq \frac{C}{\sqrt{T}}.\label{nu-ineq1}
\end{align}
If $-\frac{\sqrt{T}\theta}{2}<z<0$, then, $0<\theta+\frac{\eta z}{\sqrt{T}}<\theta$, since $-{\blue \theta}<\frac{z}{\sqrt{T}}<\frac{\eta z}{\sqrt{T}}$. In this case,
using \eqref{diff-proba-estimate},  \eqref{MVT-eta} and $\frac{g^{\prime}\left(\theta+\frac{\eta z}{\sqrt{T}}\right)}{g^{\prime}(\theta)}>1$, we obtain
\begin{align*}\left|\mathbb{P} \left\{N\geq \nu\right\}-\phi(z)\right|&=\left|\phi\left(\frac{z}{g^{\prime}(\theta)} g^{\prime}\left(\theta+\frac{\eta z}{\sqrt{T}}\right)\right)-\phi(z)\right|\\
& \leq \frac{|z|}{\sqrt{2\pi}}\left(\frac{g^{\prime}\left(\theta+\frac{\eta z}{\sqrt{T}}\right)}{g^{\prime}(\theta)}-1\right)e^{-\frac{z^2}{2}}.
\end{align*}
Moreover, in view of \eqref{MVT-delta.}, \eqref{sup-z-exp}, $\theta+\frac{\eta\delta z}{\sqrt{T}}\geq \frac{\theta}{2}$ and $g^{\prime\prime}$ is decreasing,
\begin{align}\left|\mathbb{P} \left\{N\geq \nu\right\}-\phi(z)\right| & \leq \frac{|z|e^{-\frac{z^2}{2}}}{\sqrt{2\pi}}\frac{\eta z}{\sqrt{T}g^{\prime}(\theta)}
g^{\prime\prime}\left(\theta+\frac{\eta\delta z}{\sqrt{T}}\right)\nonumber\\
& = \frac{z^2e^{-\frac{z^2}{2}}}{\sqrt{2\pi}}\frac{\eta}{\sqrt{T}|g^{\prime}(\theta)|}
g^{\prime\prime}\left(\theta+\frac{\eta\delta z}{\sqrt{T}}\right)\nonumber\\
& \leq \frac{z^2e^{-\frac{z^2}{2}}}{\sqrt{2\pi}}\frac{\eta}{\sqrt{T}|g^{\prime}(\theta)|}
g^{\prime\prime}\left(\frac{\theta}{2}\right)\nonumber\\
& \leq \frac{C}{\sqrt{T}}.\label{nu-ineq2}
\end{align}
If $0\leq z \leq \sqrt{T}\theta$, then, $\theta\leq \theta+\frac{\eta z}{\sqrt{T}}\leq 2\theta$, so, 
$  {\blue 0<}\frac{g^{\prime}\left(2\theta\right)}{g^{\prime}(\theta)}\leq \frac{g^{\prime}\left(\theta+\frac{\eta z}{\sqrt{T}}\right)}{g^{\prime}(\theta)}\leq1$. In this situation,
using \eqref{diff-proba-estimate} and  \eqref{MVT-eta}, we get
\begin{align*}\left|\mathbb{P} \left\{N\geq \nu\right\}-\phi(z)\right|&=\left|\phi\left(z\frac{g^{\prime}\left(\theta+\frac{\eta z}{\sqrt{T}}\right)}{g^{\prime}(\theta)}\right)-\phi(z)\right|\\
& \leq \frac{|z|}{\sqrt{2\pi}}{\left |\frac{g^{\prime}\left(\theta+\frac{\eta z}{\sqrt{T}}\right)}{g^{\prime}(\theta)}-1\right|}e^{-\frac{z^2}{2}\left(\frac{g^{\prime}\left(\theta+\frac{\eta z}{\sqrt{T}}\right)}{g^{\prime}(\theta)}\right)^2}
\\
& \leq \frac{|z|}{\sqrt{2\pi}}{ \left|\frac{g^{\prime}\left(\theta+\frac{\eta z}{\sqrt{T}}\right)}{g^{\prime}(\theta)}-1\right|}e^{-\frac{z^2}{2}\left(\frac{g^{\prime}\left(2\theta\right)}{g^{\prime}(\theta)}\right)^2}.
\end{align*}
Combining  this with \eqref{MVT-delta.}, \eqref{sup-z-exp},  $\theta\leq \theta+\frac{\eta z}{\sqrt{T}}\leq 2\theta$ and $g^{\prime\prime}$ is decreasing,
\begin{align}\left|\mathbb{P} \left\{N\geq \nu\right\}-\phi(z)\right| & \leq \frac{|z|e^{-\frac{z^2}{2}\left(\frac{g^{\prime}\left(2\theta\right)}{g^{\prime}(\theta)}\right)^2}}{\sqrt{2\pi}}\frac{\eta z}{\sqrt{T}{|g^{\prime}(\theta)|}}
g^{\prime\prime}\left(\theta+\frac{\eta\delta z}{\sqrt{T}}\right)\nonumber\\
& = \frac{z^2e^{-\frac{z^2}{2}\left(\frac{g^{\prime}\left(2\theta\right)}{g^{\prime}(\theta)}\right)^2}}{\sqrt{2\pi}}\frac{\eta}{\sqrt{T}|g^{\prime}(\theta)|}
g^{\prime\prime}\left(\theta+\frac{\eta\delta z}{\sqrt{T}}\right)\nonumber\\
& \leq \frac{z^2e^{-\frac{z^2}{2}\left(\frac{g^{\prime}\left(2\theta\right)}{g^{\prime}(\theta)}\right)^2}}{\sqrt{2\pi}}\frac{\eta}{\sqrt{T}|g^{\prime}(\theta)|}
g^{\prime\prime}\left(2\theta\right)\nonumber\\
& \leq \frac{C}{\sqrt{T}}.\label{nu-ineq3}
\end{align}
If $z>\sqrt{T}\theta$, then,  $ \theta+\frac{ z}{\sqrt{T}}>2\theta>\theta$, and $-z<-\sqrt{T}\theta$. Moreover, since $g$ is  decreasing,
\[\nu=\frac{\sqrt{T}}{|g^{\prime}(\theta)|}\left(g\left(\theta+\frac{z}{\sqrt{T}}\right)-g(\theta)\right)\leq
-\frac{\sqrt{T}}{|g^{\prime}(\theta)|}\left(g(\theta)-g\left(2\theta\right)\right)<0.\]
Combining these facts with \eqref{unif-expo}, we obtain
 \begin{align}\left|\mathbb{P} \left\{N\geq \nu\right\}- \phi(z)\right|&= \left|\phi(\nu)-\phi(-z)\right| \nonumber\\
& \leq\phi(\nu)+ \phi(-z)\nonumber\\
& \leq  \phi\left(-\frac{\sqrt{T}}{|g^{\prime}(\theta)|}\left(g(\theta)-g\left(2\theta\right)\right)\right) +\phi\left(-\sqrt{T}\theta\right)\nonumber
 \\&\leq \frac{C}{\sqrt{T}}.\label{nu-ineq4}
\end{align} 
As a consequence, combining \eqref{nu-ineq1}$-$\eqref{nu-ineq4}, we get \eqref{phi-nu}. Thus the proof of Lemma \ref{key-lemma2} is complete.
\end{proof}
\begin{remark}The case when the function $f$ in Lemma \ref{key-lemma2} is of the form $f(x)=Cx^{-\frac{1}{\alpha}}$ with $0<\alpha<2$, was proved in \cite{CZ}
\end{remark}

\begin{theorem}\label{dKol-rate-CLT-Vn}
Assume that the assumption  $\mathbf{\left(\mathcal{A}\right)}$ holds. Let $f$ and $g$ be  two functions defined on $(0,\infty)$ such that, for every $n\geq1$,  $f(v_n(X))>0$ almost surely, $g(x)=f^{-1}(x)$,   $g^{\prime}(x)<0$, $g^{\prime\prime}(x)>0$, $g^{\prime\prime\prime}(x)<0$. Then, there exists $C>0$ depending only on $\theta$ and $\rho(0)$ such that for all
$n\geq1$,
\begin{eqnarray}
 d_{Kol}\left(\frac{\sqrt{T_n}\left(f(v_n(X))-f\left(\E(Z_0^2)\right)\right)}{\sigma|f^{\prime}\left(\E(Z_0^2)\right)|},\mathcal{N}(0,1)\right)\nonumber &\leq& \frac{C}{\sqrt{T_n}}+C\psi_n,\label{dKol-inverse-V(X)}
\end{eqnarray}
where $\psi_n$ is given in \eqref{defi-psi}.
\end{theorem}
\begin{proof}Using Lemma \ref{key-lemma2} and the fact that $g^{\prime}\left[f\left(\E(Z_0^2)\right)\right]={1}/{f^{\prime}\left(\E(Z_0^2)\right)}$, we can write
\begin{align} 
&d_{Kol}\left(\frac{\sqrt{T_n}}{\sigma|f^{\prime}\left(\E(Z_0^2)\right)|}\left(f(v_n(X))-f\left(\E(Z_0^2)\right)\right),\mathcal{N}(0,1)\right)\nonumber \\&\leq
			  C  \left[d_{Kol}\left(\frac{\sqrt{T_n}}{\sigma}\left(v_n(X)-\E v_n(X)\right), \mathcal{N}(0,1)\right) +\sqrt{T}\left| \E v_n(X) - \E(Z_0^2)\right|+\frac{1}{\sqrt{T_n}}\right]
\nonumber \\&\leq
			  C  \left[d_{Kol}\left(\frac{1}{\sigma}V_n(X), \mathcal{N}(0,1)\right) +\sqrt{T}\left| \E v_n(X) - \E(Z_0^2)\right|+\frac{1}{\sqrt{T_n}}\right],\label{dKol-estimate1} 
		\end{align}
where the latter inequality comes from \eqref{ineq-mean-dKol}.\\
Next, combining \eqref{error X-Z}, \eqref{dKol-V(X)} and \eqref{dKol-estimate1}, the estimate  \eqref{dKol-inverse-V(X)} is obtained.
\end{proof}

\begin{theorem}\label{dW-rate-CLT-Vn}
Assume that the assumption  $\mathbf{\left(\mathcal{A}\right)}$ holds. Let $f : \R^+ \to \R^+$ be an invertible function that is a diffeomorphism, and is twice continuously differentiable. Then,   there exists $C>0$ depending only on $\theta$ and $\rho(0)$   such that for large
$n\geq1$,
\begin{eqnarray}
d_{W}\left(\frac{\sqrt{T_n}\left(f(v_n(X))-f\left(\E(Z_0^2)\right)\right)}{\sigma|f^{\prime}\left(\E(Z_0^2)\right)|},\mathcal{N}(0,1)\right)  \leq C\left[\frac{\|f^{\prime\prime}(\zeta_{n})\|_{L^2(\Omega)}}{\sqrt{T_n}} +\frac{1}{\sqrt{T_n}}+\psi_n\right],\label{dW-inverse-V(X)}
\end{eqnarray}
where $\psi_n$ is given in \eqref{defi-psi}, and $\zeta_n$ is a random variable between $v_n(X)$ and  $\E(Z_0^2)$.
\end{theorem}
\begin{proof} We have, for every $n\geq1$,
\begin{align*}
 f(v_n(X))-f\left(\E(Z_0^2)\right) =f^{\prime}\left(\E(Z_0^2)\right)\left(v_n(X) - \E(Z_0^2)\right)+\frac{1}{2} f^{\prime\prime}(\zeta_{n})\left(v_n(X) - \E(Z_0^2)\right)^2
\end{align*}
for some random point $\zeta_{n}$ between $v_n(X)$ and  $\E(Z_0^2)$.\\
Hence
\begin{align}
&d_{W}\left(\frac{\sqrt{T_n}}{\sigma|f^{\prime}\left(\E(Z_0^2)\right)|}\left(f(v_n(X))-f\left(\E(Z_0^2)\right)\right),\mathcal{N}(0,1)\right)\nonumber\\
&\leq
\frac{1}{2} {\frac{\sqrt{T_n}}{\sigma|f^{\prime}\left(\E(Z_0^2)\right)|}}\E \left| f^{\prime\prime}(\zeta_{n})\left(v_n(X) - \E(Z_0^2)\right)^2  \right|
  + d_{W}\left(\frac{1}{\sigma} V_n(X),\mathcal{N}(0,1)\right).\label{estimate1-dW}
\end{align}
The last term in this  inequality is bounded in \eqref{dW-V(X)}. Combining H\"older's inequality,  the hypercontractivity property \eqref{hypercontractivity}, \eqref{error X-Z} and \eqref{variance-V(Z)}, we obtain
\begin{eqnarray}
\E \left| f^{\prime\prime}(\zeta_{n})\left(v_n(X) - \E(Z_0^2)\right)^2  \right| &\leq  &  \left( \E |f^{\prime\prime}(\zeta_{n})|^2\right)^{1/2} \left( \E \left|v_n(X) - \E(Z_0^2) \right|^{4}  \right)^{1/2} \nonumber \\
&\leq & \quad C\left( \E |f^{\prime\prime}(\zeta_{n})|^2\right)^{1/2} \E \left|v_n(X) - \E(Z_0^2) \right|^2\nonumber 
\\ &\leq &
\frac{C\|f^{\prime\prime}(\zeta_{n})\|_{L^2(\Omega)}}{T_n}  \E \left[V_n(X)^2\right]\nonumber 
\\ &\leq &
\frac{C\|f^{\prime\prime}(\zeta_{n})\|_{L^2(\Omega)}}{T_n} \label{estimate2-dW}.
\end{eqnarray}
Therefore, in view of \eqref{dW-V(X)} and \eqref{estimate1-dW} and \eqref{estimate2-dW}, the bound~\eqref{dW-inverse-V(X)} follows.

\end{proof}

 \section{\textbf{Applications to Gaussian Ornstein-Uhlenbeck processes }}\label{applications}
In this section we  apply our general method from Section \ref{sect-main} to specific examples,
we consider the fractional Ornstein-Uhlenbeck  processes of the first kind and of the second kind.

\subsection{Fractional Ornstein-Uhlenbeck process of the first kind}\label{FOUsect}

Here we consider the fractional Ornstein-Uhlenbeck process of the first kind $X^{\theta} := \left\{X^{\theta}_{t},t\geq 0\right\} $, defined as  
  solution of the following linear stochastic differential equation
\begin{equation}
X^{\theta}_{0}=0;\quad dX^{\theta}_{t}=-\theta X^{\theta}_{t}dt+dB_{t}^{H},\quad t\geq 0, \label{FOU}
\end{equation}
where $\theta >0$ is an unknown parameter   which we would like to estimate, and  $\left\{B_{t}^{H},t\geq 0\right\} $ is a fractional Brownian motion of Hurst index $H\in (0,1)$.
It is well known that the equation~\eqref{FOU} has an explicit solution:
\begin{align}
X^{\theta}_{t}=\int_{0}^{t}e^{-\theta (t-s)}dB_{s}^{H}. \label{fOUX}
\end{align}
Moreover,
\begin{equation}
Z_{t}^{\theta }=\int_{-\infty }^{t}e^{-\theta (t-s)}dB_{s}^{H}.\label{Ztheta}
\end{equation}
 is a stationary Gaussian process, see \cite{CKM,EV}.\\
 The following theorem   provides a Berry-Esseen bound 
in Kolmogorov and   Wasserstein distances  for the   estimator $f_H\left(v_n\left(X^{\theta}\right)\right)$ of $\theta$, where 
\begin{equation}v_n\left(X^{\theta}\right) = \frac{1}{n} \sum_{i =0}^{n-1}\left[X_{t_{i}}^{\theta}\right]^{2},\, \mbox{and } f_H(x)=\left(\frac{H\Gamma(2H)}{x}\right)^{\frac{1}{2H}},\, x>0.\label{defi-fH}
\end{equation}
 \begin{theorem} Let $H\in\left(0,\frac34\right)$. Denote $\sigma_H^2:= 4 \int_{\R}\rho_{Z^{\theta}}^2(r)dr$ with $\rho_{Z^{\theta}}(r)=\E\left[Z_r^{\theta}Z_0^{\theta}\right],\, r~>~0$. Then there  exists $C>0$ depending only on $\theta$ and $H$   such that for large
$n\geq1$,
\begin{align}
 d_{Kol}\left(\frac{2H^2\Gamma(2H)\sqrt{T_n}\left(f_H\left(v_n\left(X^{\theta}\right)\right)-\theta\right)}{\sigma_H\theta^{2H+1}},\mathcal{N}(0,1)\right)
 &\leq C  \left\{ \begin{array}{ll}
\Delta_n+\frac{1}{\sqrt{n\Delta_n}} & \mbox{ if }0<H \leq \frac58, \\ ~~ &  \\
\Delta_n+\frac{1}{(n\Delta_n)^{3-4H}} & \mbox{ if }\frac58<H<\frac34,
\end{array} \right.  \label{dKol-FOU}
\end{align}
and 
\begin{align}
 d_{W}\left(\frac{2H^2\Gamma(2H)\sqrt{T_n}\left(f_H\left(v_n\left(X^{\theta}\right)\right)-\theta\right)}{\sigma_H\theta^{2H+1}},\mathcal{N}(0,1)\right)
 &\leq C  \left\{ \begin{array}{ll}
\Delta_n+\frac{1}{\sqrt{n\Delta_n}} & \mbox{ if }0<H \leq \frac58, \\ ~~ &  \\
\Delta_n+\frac{1}{(n\Delta_n)^{3-4H}} & \mbox{ if }\frac58<H<\frac34,
\end{array} \right. \label{dW-FOU}
\end{align}
 where $v_n\left(X^{\theta}\right)$ and $f_H$ are given in \eqref{defi-fH}.
  \end{theorem}
 \begin{proof}
  According to \cite[Lemma 3.1]{AEA25},  the  assumption $\mathbf{\left(\mathcal{A}\right)}$ holds with 
$Z=Z^{\theta }$ and $\gamma=2-2H$. Moreover, in this situation, we have 
 \begin{equation}
\rho(0)=\rho_{Z^{\theta}}(0) = \theta^{-2H} H \Gamma(2H)=f_H^{-1}(\theta),\quad \psi_n=\psi_n^{\theta}:= \Delta_n+\left\{ \begin{array}{ll} \frac{1}{\sqrt{T_n}} & \text { if } 0<H\leq\frac58, \\
&\\
T_n^{4H-3} & \text { if } \frac58<H<\frac34, \end{array} \right.\label{psi-FOU}
\end{equation}
where $\psi_n$ is given in \eqref{defi-psi} with $\gamma=2-2H$.\\
Now, using Theorem \ref{dKol-rate-CLT-Vn} with $f(x)=f_H(x)$, and \eqref{psi-FOU} combined with the fact that $f_H\left(\E[(Z_0^{\theta} )^2]\right)=\theta$ and $|f^{\prime}\left(\E[(Z_0^{\theta} )^2]\right)|=\frac{\theta^{2H+1}}{2H^2\Gamma(2H)}$, we obtain, for large
$n\geq1$,
\begin{align*}
 d_{Kol}\left(\frac{2H^2\Gamma(2H)\sqrt{T_n}\left(f_H(v_n\left(X^{\theta}\right))-\theta\right)}{\sigma\theta^{2H+1}},\mathcal{N}(0,1)\right)\nonumber &\leq \frac{C}{\sqrt{T_n}}+C\psi_n^{\theta}\nonumber\\
 &\leq C  \left\{ \begin{array}{ll}
\Delta_n+\frac{1}{\sqrt{n\Delta_n}} & \mbox{ if }0<H \leq \frac58, \\ ~~ &  \\
\Delta_n+\frac{1}{(n\Delta_n)^{3-4H}} & \mbox{ if }\frac58<H<\frac34,
\end{array} \right.  
\end{align*}
which implies \eqref{dKol-FOU}. \\
Similarly, using Theorem \ref{dW-rate-CLT-Vn} with $f(x)=f_H(x)$,  and \eqref{psi-FOU}, we have, for large
$n\geq1$,
\begin{align*}
 d_{W}\left(\frac{2H^2\Gamma(2H)\sqrt{T_n}\left(f_H(v_n\left(X^{\theta}\right))-\theta\right)}{\sigma\theta^{2H+1}},\mathcal{N}(0,1)\right)\nonumber &\leq C\left[\frac{\|f_H^{\prime\prime}(\zeta_{n})\|_{L^2(\Omega)}}{\sqrt{T_n}} +\frac{1}{\sqrt{T_n}}+ \psi_n \right].
\end{align*}
 Moreover, since $|f_H^{\prime\prime}|$ is decreasing and $\zeta_n$, in this case, is a random variable between $v_n\left(X^{\theta}\right)$ and  $\E[(Z_0^{\theta} )^2]$, we have $\|f_H^{\prime\prime}(\zeta_{n})\|_{L^2(\Omega)}\leq C \left[1+\|f_H^{\prime\prime}\left(v_n\left(X^{\theta}\right)\right)\|_{L^2(\Omega)}\right]$. Furthermore, 
according to the proof  of \cite[Theorem 5.4]{DEKN}, $\sup_{n\geq n_0}\|f_H^{\prime\prime}\left(v_n\left(X^{\theta}\right)\right)\|_{L^2(\Omega)}  < \infty$ for some $n_0\geq1$. Thus,  for large
$n\geq1$,
\begin{align*}
 d_{W}\left(\frac{2H^2\Gamma(2H)\sqrt{T_n}\left(f_H(v_n\left(X^{\theta}\right))-\theta\right)}{\sigma\theta^{2H+1}},\mathcal{N}(0,1)\right)
 &\leq C  \left\{ \begin{array}{ll}
\Delta_n+\frac{1}{\sqrt{n\Delta_n}} & \mbox{ if }0<H \leq \frac58, \\ ~~ &  \\
\Delta_n+\frac{1}{(n\Delta_n)^{3-4H}} & \mbox{ if }\frac58<H<\frac34.
\end{array} \right. 
\end{align*}
Therefore \eqref{dW-FOU} follows.
\end{proof}
\begin{remark}  It is interesting to note that   the upper bound in \eqref{dKol-FOU} is strictly sharper than the following  \begin{align*}
    \left[ n\Delta_n^{2H+1} \right]^{1/2}+   \left\{ \begin{array}{ll} \frac{1}{ \sqrt{n \Delta_n} } & \mbox{ if } 0<H \leq \frac{5}{8},  \\
~~ &  \\
\frac{1}{(n\Delta_n)^{3-4H}} & \mbox{ if }\frac{5}{8}<H<\frac{3}{4}, \end{array}\right.
\end{align*}
 obtained in \cite[Theorem 5.4.]{DEKN},  due to the fact that, for every $H\in(0,1)$, $\Delta_n^{2-2H}< n\Delta_n$, as $n \to \infty$. However,  the estimate \eqref{dW-FOU} has   already been obtained in \cite[Theorem 3.7]{AEA25}. 
\end{remark}

\subsection{ Fractional Ornstein-Uhlenbeck process of the second kind\label{FOUSKsection}}

In this section we consider  the so-called fractional Ornstein-Uhlenbeck process of the second kind, defined as follows
\begin{equation}
S_{0}^{\mu }=0,\mbox{ and }\ dS_{t}^{\mu }=-\mu S_{t}^{\mu }dt+dY_{t}^{(1)},\quad t\geq 0, \label{FOUSK}
\end{equation}%
where $Y_{t}^{(1)}=\int_{0}^{t}e^{-s}dB^H_{a_{s}}$ with $a_{s}=He^{\frac{s}{H}} $ and $\left\{ B^H_{t},t\geq 0\right\} $ is a fractional Brownian motion with Hurst parameter $H\in \left(\frac{1}{2},1\right)$,    whereas $\mu >0$ is the unknown  parameter. In view of  \cite[Equation 3.9]{KS}, the equation (\ref{FOUSK}) has an explicit solution:
\begin{equation*}
S_{t}^{\mu }=e^{-\mu t}\int_{0}^{t}e^{\mu s}dY_{s}^{(1)}=e^{-\mu t}\int_{0}^{t}e^{(\mu -1)s}dB^H_{a_{s}}=H^{(1-\mu )H}e^{-\mu t}\int_{a_{0}}^{a_{t}}r^{(\mu -1)H}dB^H_{r}.
\end{equation*}
Hence we can also write
\begin{equation*}
S_{t}^{\mu }=Z_{t}^{\mu }-e^{-\mu t}Z_{0}^{\mu },
\end{equation*}
where
\begin{equation*}
Z_{t}^{\mu }=e^{-\mu t}\int_{-\infty }^{t}e^{(\mu -1)s}dB^H_{a_{s}}=H^{(1-\mu )H}e^{-\mu t}\int_{0}^{a_{t}}r^{(\mu -1)H}d B^H_{r}.
\end{equation*}
Moreover,  $\left\{Z_t^\mu,t\geq 0\right\} $ is a centered stationary Gaussian process, see, for instance,  \cite[Theorem 3.8]{khalifa-hermite}. 
Furthermore, it follows from \cite[Lemma 37]{EV} that, for every $H\in (\frac{1}{2},1)$,
\begin{equation}
\rho_{Z^{\mu}}(0):=E\left[ \left( Z_{0}^{\mu }\right) ^{2}\right] =\frac{(2H-1)H^{2H}}{\mu} \mathcal{B} (1-H+\mu H,2H-1)=:f_{\mu}^{-1}(\mu),\label{defi-f-mu}
\end{equation}%
where   $\mathcal{B}(\cdot,\cdot)$ is the usual beta function. 

We will need the following lemma.
\begin{lemma}\label{behavior-rho-FOUSK} Let 
    $H \in(0,1) \backslash\left\{\frac{1}{2}\right\}$. Then there exists a constant $C>0$ depending only on $\theta$ and $H$ such that, for large  $t >2$,
      \begin{align}\max\left(|\rho_{Z^{\mu }}(t)|,|\rho_{Z^{\mu }}^{\prime}(t)|\right)\leq    C e^{-\frac12\min
\left(\mu,\frac{1}{H}-1\right) t}.\label{behavior-rho-prime-FOUSK} 
     \end{align}  
  \end{lemma}
  \begin{proof}  
  According to \cite{KS}, there exist $c,C>0$ such that for all large $t>2$,
\begin{equation*}
\left|\rho_{Z^{\mu }}(t)\right|= \left|\E\left[ Z_{0}^{\mu }Z_{t}^{\mu }\right]\right| \leq Ce^{-c|t|}.
\end{equation*}
Thus, it remains to study $|\rho_{Z^{\mu }}^{\prime}(t)|$.
 In view of \cite[Theorem 3.1]{khalifa-hermite}, we have, for very $t\geq0$,
  \begin{align}\rho_{Z^{\mu }}(t)&=e^{-\mu t}\rho_{Z^{\mu }}(0)+(2H-1)H^{2H-1}\frac{e^{-\mu t}}{2}\int_{0}^{t}\int_{-\infty}^{0}
e^{\mu u} e^{\mu v}
\left(e^{\frac{v-u}{2H}}-e^{-\frac{v-u}{2H}}\right)^{2H-2}dudv\nonumber\\
&=:e^{-\mu t}\rho_{Z^{\mu }}(0)+\frac12 (2H-1)H^{2H-1} h(t).\label{eq1}
\end{align}
  From the proof of  \cite[Lemma 2.7]{khalifa-hermite}, we have, for every  $H\in(0,\frac12)\cup(\frac12,1)$ and $t>2$,
\begin{eqnarray*}
h(t)&=&e^{-\mu t}\int_0^t\int_{-\infty}^0e^{\mu x}e^{\mu
y}\left(e^{\frac{y-x}{2H}}-
e^{-\frac{y-x}{2H}}\right)^{2H-2}dxdy\nonumber\\
&=&\frac{e^{-\mu t}}{2\mu}\int_0^{t}e^{-\mu
u}\left(e^{\frac{u}{2H}}-
e^{-\frac{u}{2H}}\right)^{2H-2}\left(e^{2\mu
u}-1\right)du +\frac{e^{\mu
t}}{2\mu}\int_t^{\infty}e^{-\mu
u}\left(e^{\frac{u}{2H}}-
e^{-\frac{u}{2H}}\right)^{2H-2}du\nonumber
\\
&&-\frac{e^{-\mu
t}}{2\mu}\int_t^{\infty}e^{-\mu
u}\left(e^{\frac{u}{2H}}-
e^{-\frac{u}{2H}}\right)^{2H-2} du\nonumber\\
&=:&m(t)+l(t)-k(t).
\end{eqnarray*}
  Moreover, there exists a constant $C>0$ depending only on $\mu$ and $H$ such that, for large  $t >2$,
 \begin{equation}\max\left(|m(t)|, |l(t)|,|k(t)|\right)\leq  C e^{-\frac12\min
\left(\mu,\frac{1}{H}-1\right) t}.\label{eq2}\end{equation}
On the other hand,   for every  $H\in(0,\frac12)\cup(\frac12,1)$ and $t>2$,
\begin{eqnarray*}
 h^{\prime}(t)
&=&-\mu m(t)+\frac{1}{2\mu} \left(e^{\frac{t}{2H}}-
e^{-\frac{t}{2H}}\right)^{2H-2}\left(1-e^{-2\mu
t}\right) +\mu l(t)  -\frac{1}{2\mu}\left(e^{\frac{t}{2H}}-
e^{-\frac{t}{2H}}\right)^{2H-2} \nonumber
\\
&&+\mu k(t)  +\frac{e^{-2\mu
t}}{2\mu}\left(e^{\frac{t}{2H}}-
e^{-\frac{t}{2H}}\right)^{2H-2}.
\end{eqnarray*}
Combining this with \eqref{eq2} and the fact that  $\left(e^{\frac{t}{2H}}-
e^{-\frac{t}{2H}}\right)^{2H-2} \sim e^{-\left(\frac{1}{H}-1\right) t},\quad
\mbox{ as } t\rightarrow\infty$, there exists a constant $C>0$ depending only on $\mu$ and $H$ such that, for large  $t >2$,
 \begin{equation} |h^{\prime}(t)| \leq  C e^{-\frac12\min
\left(\mu,\frac{1}{H}-1\right) t}.\label{eq3}\end{equation}
  Therefore, using \eqref{eq1} and \eqref{eq3}, the proof of \eqref{behavior-rho-prime-FOUSK} is complete.
   \end{proof}

 Now, we are ready to derive a Berry-Esseen bound 
in Kolmogorov and   Wasserstein distances  for the   estimator $f_{\mu}\left(v_n\left(S^{\mu}\right)\right)$ of $\mu$, with  $f_{\mu}$ is given in \eqref{defi-f-mu}, and
\[v_n\left(S^{\mu}\right) = \frac{1}{n} \sum_{i =0}^{n-1}\left[S_{t_{i}}^{\mu}\right]^{2}.\]
 \begin{theorem}  Denote $\sigma_{\mu}^2:= 4 \int_{\R}\rho_{Z^{\mu}}^2(r)dr$ with $\rho_{Z^{\mu}}(r)=\E\left(Z_r^{\mu}Z_0^{\mu}\right)$. If $H\in(\frac12,1)$, then there  exists $C>0$ depending only on $\mu$ and $H$   such that for large
$n\geq1$,
\begin{align}
 d_{Kol}\left(\frac{\sqrt{T_n}\left(f_{\mu}\left(v_n\left(S^{\mu}\right)\right)-\mu\right)}{\sigma_{\mu}\left|f_{\mu}^{\prime}\left(\rho_{Z^{\mu}}(0)\right)\right|},\mathcal{N}(0,1)\right)
 &\leq C  \left(\Delta_n+\frac{1}{\sqrt{n\Delta_n}}\right),  \label{dKol-FOUSK}
\end{align}
and 
\begin{align}
 d_{W}\left(\frac{\sqrt{T_n}\left(f_{\mu}\left(v_n\left(S^{\mu}\right)\right)-\mu\right)}{\sigma_{\mu}\left|f_{\mu}^{\prime}\left(\rho_{Z^{\mu}}(0)\right)\right|},\mathcal{N}(0,1)\right)
 &\leq C  \left(\Delta_n+\frac{1}{\sqrt{n\Delta_n}}\right),  \label{dW-FOUSK}
\end{align}
 where  $f_{\mu}$ is given in \eqref{defi-f-mu}.
  \end{theorem}

\begin{proof} 
In view of Lemma \ref{behavior-rho-FOUSK}  According to \cite[Lemma 3.1]{AEA25},  the  assumption $\mathbf{\left(\mathcal{A}\right)}$ holds with 
$Z=Z^{\mu }$ and every $\gamma>0$. In this situation, using \eqref{defi-f-mu}, we have 
 \begin{equation}
\rho(0)=\rho_{Z^{\mu}}(0)=f_{\mu}^{-1}(\mu)=:g_{\mu}(\mu).\label{rho-FOUSK}
\end{equation}
Moreover, the  assumption $\mathbf{\left(\mathcal{A}\right)}$ holds for  every $\gamma>0$, we can choose $\gamma>\frac34$, and then, the function $\psi_n$ defined by \eqref{defi-psi}, becomes
\begin{equation}
 \psi_n=\psi_n^{\mu}:= \Delta_n+  \frac{1}{\sqrt{T_n}}. \label{psi-FOUSK}
\end{equation}
On the other hand, by straightforward calculations, the function $g_{\mu}(x)$ satisfies, for all $x>0$, 
\begin{equation}g_{\mu}^{\prime}(x)<0,\, g_{\mu}^{\prime\prime}(x)>0,\, g_{\mu}^{\prime\prime\prime}(x)<0.\label{cond-g-mu}
\end{equation}
Next, applying Theorem \ref{dKol-rate-CLT-Vn} with $f(x)=f_{\mu}(x)$ and using $f_{\mu}\left(\E[(Z_0^{\mu} )^2]\right)=\mu$, \eqref{psi-FOUSK} 
and \eqref{cond-g-mu}, we deduce that, for large
$n\geq1$,
\begin{align*}
 d_{Kol}\left(\frac{\sqrt{T_n}\left(f_{\mu}\left(v_n\left(S^{\mu}\right)\right)-\mu\right)}{\sigma_{\mu}\left|f_{\mu}^{\prime}\left(\rho_{Z^{\mu}}(0)\right)\right|},\mathcal{N}(0,1)\right)
 &\leq \frac{C}{\sqrt{T_n}}+C\psi_n^{\mu}\nonumber\\
 &\leq C  \left(\Delta_n+\frac{1}{\sqrt{n\Delta_n}}  \right), 
\end{align*}
which proves \eqref{dKol-FOUSK}. \\
Similarly, applying Theorem \ref{dW-rate-CLT-Vn} with $f(x)=f_{\mu}(x)$, we have, for large
$n\geq1$,
\begin{align*}
 d_{W}\left(\frac{\sqrt{T_n}\left(f_{\mu}\left(v_n\left(S^{\mu}\right)\right)-\mu\right)}{\sigma_{\mu}\left|f_{\mu}^{\prime}\left(\rho_{Z^{\mu}}(0)\right)\right|},\mathcal{N}(0,1)\right)
 &\leq C\left[\frac{\|f_{\mu}^{\prime\prime}(\zeta_{n})\|_{L^2(\Omega)}}{\sqrt{T_n}} +\frac{1}{\sqrt{T_n}}+\psi_n^{\mu}\right]\\
 &\leq C\left[\frac{\|f_{\mu}^{\prime\prime}(\zeta_{n})\|_{L^2(\Omega)}}{\sqrt{T_n}} +\Delta_n+\frac{1}{\sqrt{T_n}}\right],
\end{align*}
where, in this case, $\zeta_n$ is a random variable between $v_n\left(S^{\mu}\right)$ and  $\E[(Z_0^{\mu} )^2]$.\\
To achieve \eqref{dW-FOUSK}, it is sufficient to prove that,  for some $n_0\geq1$, 
\begin{align}\sup_{n\geq n_0}\|f_{\mu}^{\prime\prime}(\zeta_{n})\|_{L^2(\Omega)}< \infty. \label{f-ineq}
\end{align}
In view of \eqref{defi-f-mu} and \eqref{rho-FOUSK}, we have, for all $x>0$,
\begin{align*}
g_{\mu}(x)=\frac{(2H-1)H^{2H}}{x}\int_0^1 t^{xH-H}(1-t)^{2H-2}dt,
\end{align*}
\begin{align*}
g_{\mu}^{\prime}(x)= (2H-1)H^{2H}   \int_0^1\left(-\frac{1}{x^2}+\frac{H\log(t) }{x}\right) t^{xH-H}(1-t)^{2H-2}dt, 
\end{align*}
and
\begin{align*}
g_{\mu}^{\prime\prime}(x)= (2H-1)H^{2H}   \int_0^1\left(\frac{2}{x^3}-\frac{2H\log(t) }{x^2}+\frac{H^2\log^2(t) }{x}\right) t^{xH-H}(1-t)^{2H-2}dt.
\end{align*}
Observe that, for all $x>0$, 
\begin{align}
g_{\mu}(x)>0,\quad g_{\mu}^{\prime}(x)<0,\quad g_{\mu}^{\prime\prime}(x)>0, \quad |g_{\mu}^{\prime}(x)|\geq \frac{|g_{\mu}(x)|}{x}. \label{ineq1-g-mu-FOUSK}
\end{align}
Moreover, there exists a constant $C>0$ depending only on $H$ such that, for all $x>0$, 
\begin{align}
g_{\mu}(x)\leq \frac{C}{x},\quad f_{\mu}(x)\leq \frac{C}{x},\quad \left|g_{\mu}^{\prime\prime}(x)\right|\leq C\left(\frac{1}{x^3}+\frac{1}{x^2}+\frac{1}{x}\right). \label{ineq2-g-mu-FOUSK}
\end{align}
Combining \eqref{ineq1-g-mu-FOUSK}, \eqref{ineq2-g-mu-FOUSK} and the fact that 
\[f_{\mu}^{\prime\prime}(x)=-\frac{g_{\mu}^{\prime\prime}(f_{\mu}(x))}{\left(g_{\mu}^{\prime}(f_{\mu}(x))\right)^3},\]
we deduce 
\begin{align*}
\left|f_{\mu}^{\prime\prime}(\zeta_{n})\right|=\frac{\left|g_{\mu}^{\prime\prime}(f_{\mu}(\zeta_{n}))\right|}{\left|g_{\mu}^{\prime}(f_{\mu}(\zeta_{n}))\right|^3}
&\leq C\left(\frac{1}{|f_{\mu}(\zeta_{n})|^3}+\frac{1}{|f_{\mu}(\zeta_{n})|^2}+\frac{1}{|f_{\mu}(\zeta_{n})|}\right)\frac{\left|f_{\mu}(\zeta_{n})\right|^3}{\left|\zeta_{n}\right|^3}\\
&= C\left(1+ |f_{\mu}(\zeta_{n})|+ |f_{\mu}(\zeta_{n})|^2\right)\frac{1}{\left|\zeta_{n}\right|^3}\\
&\leq C\left(\frac{1}{\left|\zeta_{n}\right|^3}+ \frac{1}{\left|\zeta_{n}\right|^4}+\frac{1}{\left|\zeta_{n}\right|^5}\right)\\
&\leq C\left(1+\frac{1}{\left|v_n\left(S^{\mu}\right)\right|^3}+ \frac{1}{\left|v_n\left(S^{\mu}\right)\right|^4}+\frac{1}{\left|v_n\left(S^{\mu}\right)\right|^5}\right).
\end{align*}
On the other hand,  according to the proof  of \cite[Theorem 5.5]{DEKN}, we have, for every $p>0$, there is $n_0 \geq 1$ such that
\[\sup_{n\geq n_0}\|\left(v_n\left(S^{\mu}\right)\right)^{-p}\|_{L^2(\Omega)}  < \infty.\]
Thus the estimate \eqref{f-ineq} is obtained, and therefore the proof of \eqref{dW-FOUSK} is complete.
\end{proof}

\begin{remark}  In \cite[Theorem 5.5]{DEKN}, the following result is proved: For $H\in(\frac12,1)$, 
\begin{eqnarray}
d_{W}\left(\frac{\sqrt{T_n}\left(f_{\mu}\left(v_n\left(S^{\mu}\right)\right)-\mu\right)}{\sigma_{\mu}\left|f_{\mu}^{\prime}\left(\rho_{Z^{\mu}}(0)\right)\right|},\mathcal{N}(0,1)\right)
 &\leq C  \left( \left[n\Delta_n^{2H+1}\right]^{1/2}+  \frac{1}{\sqrt{n\Delta_n}}  \right).\label{ex-FOUSK-DEKN}
\end{eqnarray}
In this case, we observe that the rate of
convergence provided by our bound in \eqref{dKol-FOUSK} and \eqref{dW-FOUSK} is better than  the best estimate to
date, obtained in \eqref{ex-FOUSK-DEKN},  since $\Delta_n^{2-2H}< n\Delta_n$, as $n \to \infty$.
\end{remark}

\noindent{\Large{\textbf{Declarations}}}\\

\noindent\textbf{Competing interests}\\ The authors declare no
competing interests.\\

\noindent\textbf{Authors' contributions}  All authors
read and approved the final manuscript. \\

\noindent\textbf{Funding}\\
 Y. Chen is supported by NSFC (12461029) and the PhD Research Startup Fund of Baoshan University. \\

\noindent{\textbf{Availability of data and materials}}\\
Not applicable.\\

\end{document}